\documentclass[12pt]{article}
\usepackage{geometry}
\usepackage{amsthm}
\usepackage{varioref}
\usepackage{amsmath,amstext,amsthm,a4,amssymb,amscd}   
\usepackage[T1]{fontenc}
\usepackage[utf8]{inputenc}
\usepackage{lmodern}
\usepackage{array}
\usepackage{latexsym}
\usepackage{fancyhdr}
\usepackage{mathrsfs}\let\cal\mathscr
\usepackage[all]{xy}
\usepackage{makeidx}   
\usepackage{color}
\usepackage{pifont}
\usepackage{amsfonts,amssymb}
\usepackage{hyperref}
\usepackage{cleveref}
\usepackage[numbers,square]{natbib}
\usepackage{todonotes}
\usepackage{verbatim}
\usepackage{relsize}

\usepackage{avant}
\usepackage[T1]{fontenc}      

\usepackage{amsfonts,amssymb}  
 
\usepackage{graphicx}

\usepackage{natbib}
\newcommand \Om {\Omega}
\newcommand \om {\omega}

\renewcommand \i {\sqrt{-1}}
\renewcommand \leq {\leqslant}
\renewcommand \geq {\geqslant}

\newcommand \delt {\partial_t}

\DeclareMathOperator{\Vol}{Vol}
\DeclareMathOperator{\End}{End}

\DeclareMathOperator{\Tr}{Tr}

\DeclareMathOperator{\Ker}{Ker}
\DeclareMathOperator{\rk}{rk}

\DeclareMathOperator{\Spec}{Spec}

\DeclareMathOperator{\Td}{Td}
\DeclareMathOperator{\ch}{ch}

\DeclareMathOperator{\Supp}{Supp}

\newcommand \< {\mathcal{h}}
\renewcommand \> {\mathcal{i}}

\newcommand \cinf {\CC^\infty}

\newcommand \Id {{\rm Id}}

\renewcommand \epsilon {\varepsilon}

\newcommand \CC {{\cal C}}

\newcommand \HH {{\cal H}}

\newcommand \JJ {{\cal J}}
\newcommand \K {{\mathcal K}}

\newcommand \Tau {\mathcal T}

\def\cL{\mathscr{L}}

\def\Im{{\rm Im}}

\def\cL{\mathscr{L}}

\def\cT{\mathscr{T}}

\newcommand{\til}[1]{\widetilde{#1}}

\newcommand \dt {\frac{\partial}{\partial t}}

\newcommand \D[1] {\frac{\partial}{\partial #1}}
\newcommand \Dk[2] {\frac{\partial^{#1}}{\partial {{#2}^{#1}}}}

\newcommand \R {\mathbb R}

\newcommand \C {\mathbb C}

\newcommand \N {\mathbb N}

\newcommand \Z {\mathbb Z}

\newcommand \fl {\rightarrow}

\newcommand \ignore[1] {}

\theoremstyle{plain}
\newtheorem{theorem}{Theorem}[section]
\newtheorem{lem}[theorem]{Lemma}
\newtheorem{cor}[theorem]{Corollary}
\newtheorem{prop}[theorem]{Proposition}
\theoremstyle{definition}
\newtheorem*{ackn*}{Acknowledgements}
\newtheorem{defi}[theorem]{Definition}

\numberwithin{equation}{section}

\crefname{equation}{}{}
\crefname{lem}{Lemma}{Lemmas}
\crefname{theorem}{Theorem}{Theorems}
\crefname{cor}{Corollary}{Corollaries}
\crefname{ex}{Example}{Examples}
\crefname{defi}{Definition}{Definitions}
\crefname{prop}{Proposition}{Propositions}
\crefname{section}{Section}{Sections}
\crefname{subsection}{Section}{Sections}
\crefname{rmk}{Remark}{Remarks}
\crefname{nota}{Notation}{Notations}

\hypersetup{
    colorlinks=true,
    linkcolor=red,
    filecolor=magenta,      
    urlcolor=cyan,
}

\begin{document}

\title{\bf{Geometric quantization of Hamiltonian flows and
the Gutzwiller trace formula}}
\author{Louis IOOS$^1$}
\date{}
\maketitle

\begin{abstract}

We use the theory of Berezin--Toeplitz operators of Ma and
Marinescu to study the quantum Hamiltonian dynamics associated with
classical Hamiltonian flows over closed
prequantized symplectic manifolds
in the context of geometric quantization of Kostant and Souriau.
We express the associated evolution operators via
parallel transport in the quantum spaces
over the induced path of almost complex
structures, and we establish various semi-classical estimates.
In particular, we establish a Gutzwiller trace formula for the
Kostant--Souriau operator and compute explicitly the leading term.
We then describe a potential application to contact topology.
\end{abstract}

\footnotetext[1]{Partially supported by the European Research Council Starting grant 757585}

\section{Introduction}
\label{intro}

Given a classical phase space $X$, the goal of quantization is
to produce a Hilbert space $\HH$ of quantum
states, such that the classical dynamics
over $X$, described as flows of diffeomorphisms, are mapped in a
natural way to the associated quantum dynamics of $\HH$,
described as $1$-parameter families of unitary operators.
In the context of \emph{geometric quantization}, introduced
independently by Kostant \cite{Kos70} and Souriau \cite{Sou70},
the classical phase space is represented by a $2n$-dimensional
symplectic manifold $(X,\om)$ without boundary, endowed
with a Hermitian line bundle $(L,h^L)$ together with a
Hermitian connection
$\nabla^L$ with curvature $R^L$
satisfying the following \emph{prequantization condition},
\begin{equation}\label{preq}
\om=\frac{\sqrt{-1}}{2\pi} R^L\,.
\end{equation}
The construction of an associated Hilbert space of quantum states
depends on the extra data of a \emph{polarization},
and the best choice of such
a polarization usually depends on the physical situation at hand.
In particular, the classical dynamics of a symplectic manifold
$(X,\om)$ is entirely determined by a
\emph{time-dependent Hamiltonian} $F\in\cinf(\R\times X,\R)$,
and the corresponding quantum dynamics are sometimes much easier
to infer for one specific choice of polarization.
For the general theory as well as numerous examples,
we refer to the classical book of Woodhouse
\cite[Chap.\,5,\,Chap.\,9]{Woo92}.

In this paper, we focus our attention on \emph{compact}
symplectic manifolds, and consider the polarization induced
by an almost complex structure $J\in\End(TX)$ compatible with $\om$,
which always exists. The associated Hilbert space $\HH_p$ of
quantum states will depend of an integer $p\in\N$, representing
a \emph{quantum number}, and the goal of this paper is to study
the behaviour of the quantum dynamics associated with the
classical \emph{Hamiltonian flow} of
$F\in\cinf(\R\times X,\R)$ as $p$ tends to infinity. This limit
is called the \emph{semi-classical limit}, when the scale gets
so large that we recover the laws of classical mechanics as an
approximation of the laws of quantum mechanics. We show in
particular that the quantum dynamics approximate the corresponding
classical dynamics as $p\fl+\infty$.

The construction we present in this paper holds for
any compact prequantized symplectic manifold,
and coincides with the \emph{holomorphic
quantization} of Kostant and Souriau in
the particular case when the almost complex structure $J\in\End(TX)$ is
\emph{integrable}, making $(X,J,\om)$ into a \emph{Kähler manifold}.
Then $(L,h^L)$ admits a natural holomorphic structure for which
$\nabla^L$ is its \emph{Chern connection}, and the space $\HH_p$ of
quantum states coincides with the associated space of \emph{holomorphic
sections} of the $p$th tensor power $L^p:=L^{\otimes p}$,
for all $p\in\N$ big enough. 
The natural \emph{$L^2$-Hermitian product} \cref{L2}
on the space $\cinf(X,L^p)$ of smooth sections of $L^p$ then gives
$\HH_p$ the structure of a Hilbert space.
In the very restrictive case when the Hamiltonian flow acts
by biholomorphisms on $(X,L)$,
the quantum dynamics is simply given by the induced action on
the space of holomorphic sections $\HH_p$ for all $p\in\N$.
In contrast, our results apply to the holomorphic quantization
of general Hamiltonian flows.

In \cref{setting}, we consider a general almost complex structure
$J\in\End(TX)$ compatible with $\om$, and for all $p\in\N$,
we define in \cref{ahs} the space $\HH_p$
of quantum states as the direct sum of eigenspaces associated with the
small eigenvalues of the following \emph{renormalized Bochner Laplacian}
\begin{equation}\label{Deltapintro}
\Delta_p:=\Delta^{L^p}-2\pi np\,,
\end{equation}
acting on the smooth sections $\cinf(X,L^p)$ of $L^p$, where $\Delta^{L^p}$ is the usual Bochner Laplacian of $(L^p,h^{L^p})$
associated with the Riemannian metric $g^{TX}:=\om(\cdot,J\cdot)$.
This follows an idea of Guillemin and Uribe \cite{GU88}, and we consider
here a more general construction due to Ma and Marinescu \cite{MM08a},
where $L^p$ is replaced by $E\otimes L^p$ for all $p\in\N$,
for any
Hermitian vector bundle with connection $(E,h^E,\nabla^E)$.
Note that both constructions admit an extension to the case of
a general $J$-invariant metric $g^{TX}$, and the general
construction of \cite{MM08a} also deals with the case when
one adds a potential term to $\Delta_p$.
In \cref{Approx,BTasy}, we describe the results
of \cite{Ioo18c} on the dependence of the quantization to
the choice of an almost complex structure $J\in\End(TX)$ at
the semi-classical limit $p\fl+\infty$. Specifically, we introduce
the parallel transport operator $\Tau_{p,t}$
between different quantum spaces $\HH_p$
along a path $\{J_t\in\End(TX)\}_{t\in\R}$ of almost
complex structures with respect to the \emph{$L^2$-connection}
\cref{connectionL2}, and we describe
in \cref{Approx} how $\Tau_{p,t}$
behaves like a \emph{Toeplitz operator} as
$p\fl+\infty$,
giving an explicit formula for the highest order coefficient.
This is based on the
theory of Berezin--Toeplitz operators for symplectic manifolds
developed by Ma and Marinescu in \cite{MM08b},
and extended to this context in \cite{ILMM17}.

In \cref{quantevosec}, we show how one can use this parallel transport
to construct the quantum Hamiltonian
dynamics associated with any $F\in\cinf(\R\times X,\R)$.
In fact, the corresponding Hamiltonian flow
$\varphi_t:X\fl X$ defined by \cref{Hamflow} for all $t\in\R$
does not preserves
any almost complex structure
in general, and thus does not
induce an action on $\HH_p$ for any $p\in\N$. Instead,
fix an almost complex structure $J_0\in\End(TX)$
compatible with $\om$, and consider the almost complex structure
$J_t:=d\varphi_t\,J_0\,d\varphi_t^{-1}$,
as well as the spaces $\HH_{p,t}$ of quantum states associated with
$J_t$, for all $t\in\R$ and $p\in\N$. This flow, together
with its lift to $L$ defined by \cref{liftflow}, induces a unitary
isomorphism $\varphi_{t,p}^*:\HH_{p,t}\fl\HH_{p,0}$ by pullback
on $\cinf(X,L^p)$, and considering the parallel transport
$\Tau_{p,t}:\HH_{p,0}\fl\HH_{p,t}$ along the path $s\mapsto J_s$
for $s\in[0,t]$,
the associated \emph{quantum evolution operator} at time $t\in\R$
is given by the unitary operator
$\varphi_{t,p}^*\Tau_{p,t}\in\End(\HH_{p,0})$.

In the rest of the Introduction, we assume that
the Hamiltonian $F=:f\in\cinf(X,\R)$ does not depend on
time, and write $\xi_f\in\cinf(X,TX)$ for the
\emph{Hamiltonian vector field} of $f$, as defined in
\cref{hamiltonian}. In that case, we show in \cref{quantevolem} that
\begin{equation}\label{evointro}
\varphi_{t,p}^*\Tau_{p,t}=\exp(-2\pi\i tpQ_p(f))\,,
\end{equation}
for all $t\in\R$ and $p\in\N$, where $Q_p(f)$ is the \emph{Kostant--Souriau operator} associated with $f$, defined as an operator acting on
$\cinf(X,L^p)$ by the formula
\begin{equation}\label{quanthamflaintro}
Q_p(f):=P_p\left(f
-\frac{\i}{2\pi p}\nabla^{L^p}_{\xi_f}\right)P_p\,,
\end{equation}
where $f$ is the operator of pointwise multiplication by $f$
and $P_p:\cinf(X,L^p)\fl\HH_{p,0}$ is the $L^2$-orthogonal projection.
This is the holomorphic version of the
\emph{Blattner--Kostant--Sternberg kernel}, as described for example
in \cite[\S\,9.7]{Woo92}. Under this form, it was first
noticed by Foth and Uribe in \cite[\S\,3.2]{FU07},
who interpreted the trace of $Q_p(f)$ as a
moment map for the group of Hamiltonian diffeomorphisms
acting on the space
of almost complex structures compatible with $\om$.
Using the results described in \cref{setting},
we establish the following semi-classical
estimate on its Schwartz kernel \cref{kerdef},
where $\tau_{t,p}$
denotes the parallel transport in $L^p$
along $s\mapsto\varphi_s(x)$ for $s\in\,[0,t]$.
Here and in all the paper,
we use the notation $O(p^{-k})$
in the sense of the corresponding Hermitian norm as
$p\fl+\infty$, and
$O(p^{-\infty})$ means $O(p^{-k})$ for all $k\in\N$.

\begin{prop}\label{Utthintro}
For any $t\in\R$ and $x,\,y\in X$ such that $\varphi_t(y)\neq x$,
we have the following estimate as $p\fl+\infty$,
\begin{equation}\label{thetacongUtintro}
\exp\left(2\pi\sqrt{-1}tpQ_p(f)\right)(x,y)=
O(p^{-\infty})\,.
\end{equation}
Furthermore, there exist $a_r(t,x)\in\C$ $(r\in\N)$
smooth in $x\in X$ and $t\in\R$,
such that for any $k\in\N^*$, as $p\fl+\infty$ we have
\begin{multline}\label{exppsiquant1intro}
\exp\left(2\pi\sqrt{-1}tpQ_p(f)\right)(\varphi_t(x),x)\\
=p^{n} e^{2\pi\i tpf(x)}\left(\sum_{r=0}^{k-1}
p^{-r} a_{r}(t,x) + O(p^{-k})\right)\tau_{t,p}\,,
\end{multline}
with $a_0(t,x)\neq 0$ for all $t\in\R$ and $x\in X$.
\end{prop}

This follows from the more precise \cref{Utth}, which gives
in particular a formula for the first coefficient $a_0(t,x)$.
As explained there,
this result shows that the quantum dynamics approximates the
classical dynamics at the semi-classical limit $p\fl+\infty$
in a precise sense. In \cref{sctrthcor},
we also use the results described
in \cref{setting} to
give an associated semi-classical trace formula.
Note that the estimate \cref{thetacongUtintro}
is not uniform in $t\in\R,\,x,\,y\in X$, and the
estimate \cref{exppsiquant1intro} shows that there is
in fact a jump when $\varphi_t(y)$ tends to $x$.

In \cref{gutzsec}, we consider a time-independent Hamiltonian
$f\in\cinf(X,\R)$, and we use the setting of \cref{setting}
to study the
operators $\hat{g}(pQ_p(f-c))\in\End(\HH_p)$ defined for all $p\in\N$
by the formula
\begin{equation}\label{psiquantintro}
\hat{g}(pQ_p(f-c)):=\int_\R g(t)e^{2\pi\sqrt{-1} tpc}
\left(\varphi_{t,p}^*\Tau_{p,t}\right) dt\,,
\end{equation}
where $g:\R\fl\R$ is smooth with compact support, where
$c\in\R$ is a regular value of $f$ and where
$\varphi_{t,p}^*\Tau_{p,t}\in\End(\HH_{p,0})$
is the quantum evolution operator
associated with $f$. Via the interpretation
\cref{evointro} in terms of quantum evolution operators,
the \emph{Gutzwiller trace formula} predicts a semi-classical estimate
for the trace $\Tr[\hat{g}(pQ_p(f-c))]$
as $p\fl+\infty$, in terms of the
periodic orbits of the Hamiltonian flow of $f$ included in the
level set $f^{-1}(c)$. This formula was first worked out by Gutzwiller
in \cite[(36)]{Gut71} for usual Schrödinger operators over $\R^n$
as the \emph{Planck constant} $\hbar$ tends to $0$,
using path integral methods. As explained in his book
\cite{Gut90}, this formula plays a fundamental role in the
theory of \emph{quantum chaos}, which studies the quantization
of chaotic classical systems.

Specifically, let $\Supp g\subset\R$ be the support of $g$,
and suppose that for all $t\in\Supp g$,
the fixed point set $X^{\varphi_t}\subset X$ of $\varphi_t:X\fl X$
is non-degenerate
over a neighborhood of $f^{-1}(c)$ in the sense of \cref{nondegdef}
and intersects $f^{-1}(c)$ transversally
such that $X^{\varphi_t}\cap f^{-1}(c)$ is
non-empty only for a finite subset $T\subset\Supp g$.
%Recalling that $\varphi_t$ preserves $f^{-1}(c)$ for all $t\in\R$,
Let $\{Y_j\}_{1\leq j\leq m}$ be the
set of connected components of
\begin{equation}
\coprod_{t\in T} X^{\varphi_t}\cap f^{-1}(c)\,,
\end{equation}
and
for any $1\leq j\leq m$, write $t_j\in T$ for the time such that
$Y_j\subset X^{\varphi_{t_j}}\cap f^{-1}(c)$.
In particular, these hypotheses are automatically
satisfied if $\Supp g\subset\R$ is a small enough neighborhood of $0$,
so that $T=\{0\}$ and 
$X^{\varphi_0}=X$.

Let $\lambda_j\in\R$ be
the \emph{action} of $f$ over $Y_j$ as in \cref{actiondef}, and
let $\Vol_\om(f^{-1}(c))>0$ be the volume of $f^{-1}(c)$ with
respect to the natural \emph{Liouville measure} \cref{liouvilleham}
induced by $\om$ and $f$ on $f^{-1}(c)$.
%We write $|dv|_{Y_j}$ for the Riemannian density on $Y_j$ induced by
%$g^{TX}$, and $\xi_f\in\cinf(X,TX)$
%for the \emph{Hamiltonian vector field} of $f$.
%The following Theorem is the main result of \cref{gutzsec}.

\begin{theorem}\label{gutzwillerperorbintro}
Under the above assumptions, there exist $b_{j,r}\in\C$
$(r\in\N)$, depending only on geometric data around $Y_j$
for all $1\leq j\leq m$,
%, depending only on $g$ around $0$ and on all the local data around $Y$, 
such that for any $k\in\N^*$ and as $p\fl+\infty$, we have
\begin{equation}\label{gutzexpperorbintro}
\Tr\left[\hat{g}(pQ_p(f-c))\right]=\sum_{j=1}^m
p^{(\dim Y_j-1)/2}g(t_j)
e^{-2\pi\i p \lambda_j}\left(\sum_{r=0}^{k-1}
p^{-r}b_{j,r}+O(p^{-k})\right)\,.
\end{equation}
Furthermore, there is an explicit geometric formula for the first
coefficients $b_{j,0}$ for all $1\leq j\leq m$, and
as $p\fl+\infty$ we have
\begin{equation}\label{Weylintro}
\Tr\left[\hat{g}(pQ_p(f-c))\right]=p^{n-1}g(0)
\Vol_\om(f^{-1}(c))
+O(p^{n-2})\,.
\end{equation}
\end{theorem}

Note that formula  \cref{gutzexpperorbintro} does not follow from
\cref{evointro} and \cref{Utthintro} by integrating over $t\in\R$,
due to the jump in the estimates \cref{thetacongUt} and
\cref{exppsiquant1intro} when $\varphi_t(y)$ tends to $x$.
This is illustrated by \cref{oscvener}, where it is shown
how $\hat{g}(pQ_p(f-c))$ localizes around $f^{-1}(c)$ after
integrating in $t\in\R$ via stationary phase estimates,
with a precise control on the constant around $f^{-1}(c)$.

The general formula for the first coefficients $b_{j,0}$
of the expansion
\cref{gutzexpperorbintro} is given in \cref{gutzwiller}, and
reduces to the so-called \emph{Weyl term}
\cref{Weylintro} of the trace formula in the case $0\in\Supp g$.
However, the relevance of this formula
for quantum chaos mainly lies in the terms associated with
\emph{isolated periodic orbits},
and one would like to consider general situations
where this formula exhibits natural geometric quantities associated
to such orbits.

To describe such situations, let us first consider the general case
of a Hermitian vector bundle with connection $(E,h^E,\nabla^E)$
over $\R\times X$. Writing $E_t$ for its restriction to $X$ over
$t\in\R$, we take more generally the quantum spaces $\HH_{p,t}$
to be the almost holomorphic sections of $E_t\otimes L^p$ with
respect to $J_t$, for all $p\in\N$, together with the $L^2$-Hermitian
product induced by $h^{E_t}$ and $h^{L^p}$. Following
\cref{quantbdledef}, we can again consider the parallel transport
$\Tau_{p,t}:\HH_{p,0}\fl\HH_{p,t}$ with respect to the associated
$L^2$-connection, and if $\varphi_t:X\fl X$ is a Hamiltonian flow
lifting to a bundle map $\varphi_t^E:E_0\fl E_t$ over $X$
for all $t\in\R$, we again have a unitary evolution operator
$\varphi_{t,p}^*\Tau_{p,t}\in\End(\HH_{p,0})$.
Then the right hand side of formula \cref{psiquantintro}
still makes sense,
and \cref{gutzwiller} gives the general version of
\cref{gutzwillerperorbintro} in this context.

Consider now the canonical line bundle $(K_X,h^{K_X},\nabla^{K_X})$
over $\R\times X$ associated with $\{J_t\in\End(TX)\}_{t\in\R}$
defined in \cref{setting} by the formula \cref{KXdef}, and
assume that $(X,\om)$ admits a \emph{metaplectic structure},
so that this canonical line bundle admits a square root
$K_X^{1/2}$ over $\R\times X$ with induced metric and connection,
called the \emph{metaplectic correction}.
Then $\varphi_t$ admits a natural lift for all $t\in\R$, and
we call the associated unitary operator $\varphi_{t,p}^*\Tau_{p,t}$
as above the \emph{metaplectic quantum evolution operator}.

\begin{theorem}\label{b0gutzperorbintroth}
Assume that $(X,\om)$ admits a metaplectic
structure, and consider the assumptions of
\cref{gutzwillerperorbintro}.
Let $1\leq j\leq m$ be such that $\dim Y_j=1$
and such that $[J\xi_f,\xi_f]=0$ over $Y_j$.
Then the first coefficients of the
analogous expansion \cref{gutzexpperorbintro} as $p\fl+\infty$
for the metaplectic quantum evolution operator
satisfy the following formula,
\begin{equation}\label{b0gutzperorbintro}
b_{j,0}=(-1)^{\frac{n-1}{2}}
\frac{t(Y_j)}
{|\det_{N_x}(\Id_{N}-d\varphi_{t_j}|_{N})|^{1/2}}
\,,
\end{equation}
for a natural choice of square root and for any $x\in Y_j$,
where $N$ is the normal bundle
of $Y_j$ inside $T f^{-1}(c)$ and
where $t(Y_j)>0$ is the primitive
period of $Y_j$ as a periodic orbit of the flow
$t\mapsto\varphi_t$ inside $f^{-1}(c)$.
\end{theorem}
For $(X,J_t,\om)$ Kähler for all $t\in\R$ and
endowed with a metaplectic structure,
one can show using \cite[(5.2)]{Ioo18c}
that the generator of the metaplectic quantum evolution operator
considered above coincides with the metaplectic Kostant--Souriau
operator considered by Charles in \cite[Th.\,1.5]{Cha06}.

\cref{b0gutzperorbintroth}
follows from \cref{gutzwillerperorb}, which gives also
a formula for general $(E,h^E,\nabla^E)$
as an integral along the associated periodic orbit.
In the case of usual Schrödinger operators over
a compact Riemannian manifold,
the Gutzwiller trace formula
has been established by Guillemin and Uribe \cite[Th.\,2.8]{GU89},
Paul and Uribe \cite[Th.\,5.3]{PU95} and Meinrenken
\cite[Th.\,3]{Mei92}, while in the case of Toeplitz operators
over the smooth boundariy of a compact strictly pseudoconvex domain,
it has been established by Boutet de Monvel
and Guillemin \cite[Th,\,9,\,Th.\,10]{BdMG81}.
In the case of Berezin--Toeplitz operators over a compact
prequantized Kähler manifold with metaplectic structure,
instead of Kostant--Souriau
operators over a general compact prequantized symplectic manifold
as in \cref{gutzwillerperorbintro}, it
has been established
by Borthwick, Paul and Uribe \cite[Th.\,4.2]{BPU98} using
the theory of Boutet de Monvel and Guillemin \cite{BdMG81}.

In all the works cited above, the corresponding formulas for the first
coefficients involve an undetermined
subprincipal symbol term with no obvious geometric interpretation.
In contrast, the general formula for the
first coefficient given in \cref{b0gutz} is
completely explicit in terms of local geometric
data.
Furthermore, the formula for isolated periodic orbits
given in \cref{b0gutzperorbintroth}
is the same as the corresponding formulas in
all the cases mentioned above, but without the undetermined
subprincipal symbol term.
%, as in formula \cite[(0.2)]{DG75} of Duistermaat and
%Guillemin in the case of the wave operator
%on compact Riemannian manifolds.
This makes it much simpler to use in practical applications.

In fact, let the Hamiltonian $f\in\cinf(X,\R)$ and the
almost complex structure $J\in\End(TX)$ be such that
\begin{equation}
d\iota_{J\xi_f}\om=\om~~\text{over}~~f^{-1}(I)\,,
\end{equation}
for some interval $I\subset\R$ of regular values of $f$
containing $c\in\R$.
As explained at the end of \cref{gutzsec}, this
induces a \emph{contact form} $\alpha\in\Om^1(\Sigma,\R)$
on $\Sigma:=f^{-1}(c)$, and $\xi_f$ generates the
\emph{Reeb flow} of $(\Sigma,\alpha)$. 
Then \cref{contactcor} shows how 
\cref{gutzwillerperorbintro} can be used to detect the periods
of the non-degenerate isolated periodic orbits of this flow,
and \cref{b0gutzperorbintroth} allows
in principle to compute the associated action.
This is of particular interest in contact topology, where
the study of non-degenerate isolated periodic orbits of the Reeb flow
is a major topic, usually tackled via methods of Floer homology.
Note that to recover the action from the expansion
\cref{gutzexpperorbintro} in practice, one needs a completely
explicit formula for the first coefficient, and formula
\cref{b0gutzperorbintro} is the best that one can hope for.

In \cref{exppsiquant}, we also establish 
semi-classical estimates
on the Schwartz
kernel of the operator $\hat{g}(pQ_p(f-c))$ as $p\fl+\infty$,
analogous to the
corresponding estimates in \cite[Th.\,1.1]{BPU98} for
Berezin--Toeplitz operators over compact prequantized
Kähler manifolds admitting a metaplectic structure. Once again,
our formula for the first order term
is completely explicit in terms of geometric data,
without the undetermined subprincipal symbol term appearing in the
corresponding formula in
\cite[Th.\,2.7]{BPU98}.

In the case of compact prequantized Kähler manifolds and when
the Hamiltonian flow $\varphi_t:X\fl X$ acts by biholomorphisms,
so that one can define the quantization of $\varphi_t$ simply
by its induced action on holomorphic sections,
the pointwise semi-classical estimates of 
\cref{exppsiquant} for $E=\C$
have been obtained by Paoletti in \cite[Th.\,1.2]{Pao18}.
In this case, the general
version of \cref{gutzwillerperorbintro}
is a direct consequence of the following \emph{Kirillov formula},
\begin{equation}\label{Kir}
\Tr\left[\varphi_{t_j+t,p}^*\right]=\int_{X^{\varphi_{t_j}}}
\Td_{\varphi_{t_j}^{-1},-t\xi_f}(TX)
\ch_{\varphi_{t_j}^{-1},-t\xi_f}(L^p)\,,
\end{equation}
as described in \cite[(2.38)]{BG00}
for any fixed $1\leq j\leq m$ and $|t|>0$ small enough,
using the stationary phase lemma as $p\fl+\infty$.
Paoletti recovers this special case in \cite[Th.\,1.3]{Pao18}
without using formula \cref{Kir}.
%together with the localization theorem
%of \cite[Th.\,7.13]{BGV04}.

The theory of Berezin--Toeplitz operators over compact prequantized
Kähler manifolds with $E=\C$
was first developed by Bordemann,
Meinreken and Schlichenmaier \citep{BMS94} and Schlichenmaier \citep{Sch00}. Their approach is based on
the work of Boutet de Monvel and Sjöstrand on the Szegö kernel
in \cite{BdMS75}, and
the theory of Toeplitz structures developed by Boutet de Monvel
and Guillemin in \citep{BdMG81}. The present paper is based instead on
the approach of Ma and Marinescu using Bergman kernels, and
we refer to the book \cite{MM07} for a detailed presentation of this
method.

The quantization of symplectic maps over compact prequantized
Kähler manifolds has first been considered by Zelditch in \cite{Zel97},
using a unitary version of the theory
of Toeplitz structures of \citep{BdMG81},
%, and an analogue of
%\cref{Utthintro} in this context has been established by Paoletti in
%\cite[Th.\,1.4]{Pao12}
and Zelditch and Zhou use it in
\cite[Th.\,0.9]{ZZ08} to establish the pointwise
semi-classical estimates of \cref{exppsiquant} in the Kähler case
for $E=\C$. Note that \cref{gutzwillerperorbintro} is not
a consequence of these pointwise semi-classical estimates,
as they are not uniform in $x\in X$.
The applications of parallel transport to the quantum dynamics
associated with Hamiltonian flows have also been explored by Charles
\cite[Th.\,5.3]{Cha07} in the case of compact prequantized
Kähler manifolds with
metaplectic structure, where he establishes an analogue of
\cref{Utthintro} in the
language of Fourier integral operators.

\begin{ackn*}
The author wants to thank Pr. Xiaonan Ma for his support
and Pr. Leonid Polterovich for helpful discussions.
The author also wants to thank the anonymous
referees for useful comments
and suggestions. This work was supported
by the European Research Council Starting grant 757585.
\end{ackn*}

\section{Setting}\label{setting}

Let $(X,\om)$ be a compact symplectic manifold
without boundary of dimension $2n$,
and let $(L,h^L)$ be a Hermitian line bundle over $X$, endowed with
a Hermitian connection $\nabla^L$ satisfying
the prequantization condition \cref{preq}.
Let $J$ be an almost complex structure
on $TX$ compatible with $\om$,
and let $g^{TX}$ be the Riemmanian metric on $X$ defined by
\begin{equation}
g^{TX}(\cdot,\cdot):=\om(\cdot,J\cdot)\,.
\end{equation}
%Let $\nabla^{TX}$ be the associated Levi--Civita connection,
We write $\nabla^{TX}$ for the associated Levi--Civita connection
on $TX$, and $dv_X$ for the Riemannian volume form of
$(X,g^{TX})$. It satisfies the \emph{Liouville
formula} $dv_X=\om^n/n!$.

For any Hermitian vector
bundle with connection $(E,h^E,\nabla^E)$ over $X$,
we write 
$\<\cdot,\cdot\>_E$ and $|\cdot|_E$ for the Hermitian product
and norm induced by $h^E$, and write $R^E$ for the curvature of
$\nabla^E$. We denote by $\C$ the trivial line bundle with trivial
Hermitian metric and connection.
For any $p\in\N$, we write
$L^p$ for the $p$-th tensor power of $L$, and
for any Hermitian vector bundle with connection $(E,h^E,\nabla^E)$,
we set 
\begin{equation}\label{Ep}
E_p:=L^p\otimes E\,,
\end{equation}
equipped with the Hermitian metric $h^{E_p}$ and connection
$\nabla^{E_p}$ induced by $h^L,\,h^E$ and
$\nabla^L,\,\nabla^E$.
The \emph{$L^2$-Hermitian product} $\<\cdot,\cdot\>_p$
on $\cinf(X,E_p)$ is given for any $s_1,s_2\in\cinf(X,E_p)$
by the formula
\begin{equation}\label{L2}
\<s_1,s_2\>_p:=\int_X \<s_1(x),s_2(x)\>_{E_p}\,dv_X(x)\,.
\end{equation}
Let $L^2(X,E_p)$ be the completion of $\cinf(X,E_p)$ with respect to
$\<\cdot,\cdot\>_p$.

\begin{defi}\label{bochner}
For any $p\in\N$, the \emph{Bochner Laplacian} $\Delta^{E_p}$ of
$(E_p,h^{E_p},\nabla^{E_p})$ is the second order
differential operator acting on $\cinf(X,E_p)$ by the formula
\begin{equation}\label{delta}
\Delta^{E_p}:=-\sum_{j=1}^{2n}\left[(\nabla^{E_p}_{e_j})^2
-\nabla^{E_p}_{\nabla^{TX}_{e_j}e_j}\right],
\end{equation}
where $\{e_j\}_{j=1}^{2n}$ is any local orthonormal frame of
$(TX,g^{TX})$.
\end{defi}
This defines an unbounded 
self-adjoint elliptic operator on $L^2(X,E_p)$, and
by standard elliptic theory, its spectrum $\Spec(\Delta^{E_p})$
is discrete and
contained in $\R$.
%\cite{BLY94}.\cite{Her70}

\begin{defi}\label{bochnerrenorm}
For any $p\in\N$, the \emph{renormalized Bochner Laplacian}
$\Delta_p$ is the second order differential operator acting
on $\cinf(X,E_p)$ by the formula
\begin{equation}\label{deltalpe}
\Delta_p:=\Delta^{E_p}-2\pi np-\sum_{j=1}^n R^E(w_j,\bar{w_j})\,,
\end{equation}
where $\{w_j\}_{j=1}^n$ is an orthonormal basis of
$T^{(1,0)}X$ for the Hermitian metric induced by $g^{TX}$.
\end{defi}

As above, $\Delta_p$ is an unbounded self-adjoint elliptic
operator on $L^2(X,E_p)$, and has discrete spectrum $\Spec(\Delta_p)$
contained in
$\R$. Furthermore, we have the following refinement of
\cite[Th.2.a]{GU88}.

\begin{theorem}{\cite[Cor.~1.2]{MM02}}\label{specdeltapphi}
There exist constants $\til{C},\,C>0$ such that for all $p\in\N$,
\begin{equation}\label{specdeltapphifla}
\Spec(\Delta_p)\subset\ [-\til{C},\til{C}\,]\ \cup\ ]4\pi n p-C,+\infty[\,,
\end{equation}
and the constants $\til{C},~C>0$ are uniform in the choice of
$J\in\End(TX)$ varying smoothly in a compact set of parameters.
Furthermore, the direct sum 
\begin{equation}\label{ahs}
\HH_p:=\bigoplus_{\substack{\lambda\in [-\til{C},\til{C}\,]}}
\Ker(\lambda-\Delta_p)
\end{equation}
is naturally included in $\cinf(X,E_p)$, and there is $p_0\in\N$
such that for any $p\geq p_0$, we have
\begin{equation}\label{RRH}
\dim\HH_{p}=\int_X\Td(T^{(1,0)}X)\ch(E)\exp(p\om),
\end{equation}
where $\Td(T^{(1,0)}X)$ represents the Todd class of $T^{(1,0)}X$ and $\ch(E)$ represents the Chern character of $E$.
The integer $p_0\in\N$ is uniform in the choice of $J\in\End(TX)$
varying smoothly in a compact set of parameters.
\end{theorem}

For any $p\in\N$, the Hilbert space $\HH_p\subset L^2(X,E_p)$
defined by \cref{ahs}
is called the \emph{space of almost holomorphic sections}
of $E_p$. In the special case when $J$ is integrable, so that
$(X,J,\om,g^{TX})$ is a Kähler manifold and the Hermitian bundles
$(L,h^L)$ and $(E,h^E)$
admit natural holomorphic structures such that $\nabla^L$ and
$\nabla^E$ are their Chern connections, then 
the subspace $\HH_p\subset\cinf(X,L^p)$ coincides
with the space of holomorphic sections of $E_p$ for all $p\geq p_0$.
In fact, as explained for example
in \cite[\S\,1.4.3,\,\S\,1.5]{MM07},
by the \emph{Bochner-Kodaira formula}, the formula
\cref{deltalpe} is twice
the \emph{Kodaira Laplacian} of $E_p$, and there is a
\emph{spectral gap}, so that $\til{C}=0$
in \cref{specdeltapphifla}. It is then a basic fact of Hodge theory
that the kernel of the Kodaira Laplacian in $\cinf(X,E_p)$
is precisely the
space of holomorphic sections of $E_p$, for all $p\in\N$.
%In this context, the space of $\HH_p$ of almost holomorphic sections
%of $E_p$, together with the $L^2$-Hermitian product \cref{L2},
%represents the \emph{space of quantum states} of $(X,J,\om)$ at level
%$p\in\N$.

The goal of this section is to describe the results of
\cite{Ioo18c} about the dependence of this quantization scheme on
the choice of an almost complex structure $J\in\End(TX)$.
To this end, we consider a smooth path
\begin{equation}
t\longmapsto J_t\in\End(TX),~~\text{for all}~~t\in\R\,,
\end{equation}
of almost complex structures over $X$ compatible with $\om$.
We will see $\{J_t\in\End(TX)\}_{t\in\R}$ as the endomorphism of
the vertical tangent bundle $TX$ over $\R\times X$ of the
tautological fibration
\begin{equation}\label{tautfib}
\begin{split}
\pi:\R\times X&\longrightarrow\R\\
(t,x)&\longmapsto t\,,
\end{split}
\end{equation}
restricting to $J_t\in\End(TX)$ over $t\in\R$. We then have
an induced vertical Riemannian metric on
$TX$ over $\R\times X$, defined by its restriction
on the fibre $X$ over any $t\in\R$ via the formula
\begin{equation}\label{gTXintro}
g^{TX}_t(\cdot,\cdot):=\om(\cdot,J_t\,\cdot)\,.
\end{equation}
Following Bismut in \cite[Def.\,1.6]{Bis86} and in
\cite[(1.2)]{Bis97},
we consider the induced \emph{vertical
Levi--Civita connection} $\nabla^{TX}$ on the subbundle $TX$ of
the tangent bundle of $\R\times X$, defined by the formula
\begin{equation}\label{nab=PinabPi}
\nabla^{TX}:=\Pi^{TX}\nabla^{\R\oplus TX}\Pi^{TX}\,,
\end{equation}
where $\nabla^{\R\oplus TX}$ is the Levi--Civita connection
on the total space of $\R\times X$ for the Riemannian metric
defined on $T(\R\times X)=\R\oplus TX$ as the orthogonal sum
of the canonical metric of $\R$ and the metric
$g^{TX}_t$ over $t\in\R$, with $\Pi^{TX}:\R\oplus TX\fl TX$ the
canonical projection.
Note that by the Liouville formula $dv_X=\om^n/n!$,
the Riemannian volume form $dv_X$ of $(X,g^{TX}_t)$
does not depend on $t\in\R$.

Let $TX_\C:=TX\otimes_\R \C$ be
the complexification of the vertical tangent bundle
$TX$ over $\R\times X$. The family of
complex structures $\{J_t\in\End(TX)\}_{t\in\R}$
induces a splitting
\begin{equation}\label{splitc}
TX_\C=T^{(1,0)}X\oplus T^{(0,1)}X
\end{equation}
into the eigenspaces of $J_t$ corresponding to the eigenvalues $\sqrt{-1}$ and $-\sqrt{-1}$ over $\{t\}\times X$ for all $t\in\R$.
We endow $TX_\C$ with the Hermitian product
$h^{TX}$ given by $g^{TX}_t(\cdot,\bar{\cdot})$ over $\{t\}\times X$
for all $t\in\R$.
The \emph{canonical line bundle} associated with
$\{J_t\in\End(TX)\}_{t\in\R}$ is the line bundle
\begin{equation}\label{KXdef}
K_X:=\det(T^{(1,0)*}X)
\end{equation}
over $\R\times X$ equipped with the Hermitian metric
$h^{K_X}$ and connection $\nabla^{K_X}$
induced by the vertical Hermitian metric
$h^{TX}$ and the vertical Levi--Civita connection
\eqref{nab=PinabPi} via the splitting \cref{splitc}.

%The splitting \cref{splitc} with respect to $J_t\in\End(TX)$ for all
%$t\in\R$ induces a splitting of the vertical tangent bundle $TX$
%over $\R\times X$,
%and the \emph{canonical line bundle}
%$(K_X,h^{K_X},\nabla^{K_X})$ over $\R\times X$
%is defined by \cref{KXdef}, with Hermitian metric and connection
%induced by the vertical Riemannian
%metric and the vertical Levi--Civita connection.
For any Hermitian vector bundle with connection $(E,h^E,\nabla^E)$
over $\R\times X$, we write $(E_t,h^{E_t},\nabla^{E_t})$ for the
Hermitian vector bundle with connection induced on $X$ by restriction
to the fibre over $t\in\R$. For all $t\in\R$, we write
\begin{equation}\label{tauE}
\tau^E_t:E_0\longrightarrow E_t
\end{equation}
for the bundle isomorphism over $X$ induced
by parallel transport in $E$ with respect
to $\nabla^E$ along horizontal directions of $\pi:\R\times X\fl\R$.
We still write $(L,h^L,\nabla^L)$ for the Hermitian line bundle with
connection over $\R\times X$ defined by pullback of $(L,h^L,\nabla^L)$
over $X$ via the second projection, and write
$(E_p,h^{E_p},\nabla^{E_p})$ for the tensor product $E_p=E\otimes L^p$
over $\R\times X$ for any $p\in\N$, with induced Hermitian metric
and connection.
For any $p\in\N$ and $t\in\R$, we write $\Delta_{p,t}$ for the
renormalized Bochner Laplacian acting on $\cinf(X,E_{p,t})$
associated with the metric $g^{TX}_t$ as in \cref{bochnerrenorm},
and write $\HH_{p,t}\subset\cinf(X,E_{p,t})$ for the associated
space of almost holomorphic sections defined in \cref{specdeltapphi}.

Let us assume that there exists $p_0\in\N$ such that
the two intervals in \cref{specdeltapphifla} are disjoint
and such that $\HH_{p,t}$ satisfies
the Riemann-Roch-Hirzebruch formula \cref{RRH}
for all $p\geq p_0$ and
$t\in\R$.
By \cref{specdeltapphi}, such a $p_0\in\N$ always
exists if we ask $J_t$ and $(E_t,h^{E_t},\nabla^{E_t})$ to be
independent of $t\in\R$ outside a compact set of $\R$.
On the other hand, this assumption will be automatically satisfied
for all $t\in\R$
in the main case of interest considered in \cref{quantevosec}.
In the sequel, we fix such a $p_0\in\N$.

Following for instance \cite[\S\,9.2]{BGV04},
for all $t\in\R$ and $p\geq p_0$,
we define the orthogonal projection operator
$P_{p,t}:L^2(X,E_{p,t})\fl\HH_{p,t}$ with respect to the
associated $L^2$-Hermitian product \cref{L2} by the following
contour integral in the complex plane,
\begin{equation}\label{projdef}
P_{p,t}:=\int_\Gamma \left(\lambda-\Delta_{p,t}\right)^{-1}d\lambda\,,
\end{equation}
where $\Gamma\subset\C$ is a circle of center $0$ and radius
$a>0$ satisfying $\til{C}<a<4\pi p-C$.
This shows that the projection operators
$P_{p,t}$ depend smoothly on $t\in\R$, and as the dimension of
$\Im(P_{p,t})=\HH_{p,t}$ is constant in $t\in\R$ by assumption,
this defines a finite dimensional bundle over $\R$,
which can be seen as a subbundle of the infinite dimensional
vector bundle with fibre $\cinf(X,E_{p,t})$ over
$t\in\R$.

\begin{defi}\label{quantbdledef}
For any $p\geq p_0$, the \emph{quantum bundle}
$(\HH_p,h^{\HH_p},\nabla^{\HH_p})$ is the bundle
of almost holomorphic
sections over $\R\simeq\{J_t\in\End(TX)\}_{t\in\R}$
%whose fibre over $t\in\R$ identifies with the almost holomorphic
%sections $\HH_{p,t}\subset\cinf(X,E_{p,t})$ with respect to $J_t$
defined via \cref{projdef} as above,
endowed with the \emph{$L^2$-Hermitian structure}
$h^{\HH_p}$ induced by the
$L^2$-Hermitian product of $L^2(X,E_{p,t})$
for all $t\in\R$, and with the
\emph{$L^2$-Hermitian connection} $\nabla^{\HH_p}$,
defined on the canonical
vector field $\delt$ of $\R$ via its action on the
total space $\cinf(\R\times X,E_p)$ by the formula
\begin{equation}\label{connectionL2}
\nabla^{\HH_p}_{\delt}:=P_{p,t}\nabla^{E_p}_{\delt}P_{p,t}\,,
\end{equation}
for all $t\in\R$. By convention, we set $\HH_p=\{0\}$ for all $p<p_0$.
\end{defi}
By an argument of \cite[Th.\,1.14]{BGS88}, the $L^2$-connection
$\nabla^{\HH_p}$ preserves
the $L^2$-Hermitian product $h^{\HH_p}$.
For any $p\in\N$ and $t\in\R$, let $\cL(\HH_{p,0},\HH_{p,t})$
be the space of linear operators from $\HH_{p,0}$ to $\HH_{p,t}$,
and write $\|\cdot\|_{p,0,t}$ for
the operator norm of $\cL(\HH_{p,0},\HH_{p,t})$ induced by $h^{\HH_p}$.
We consider the \emph{parallel transport}
\begin{equation}\label{Tauptdef}
\Tau_{p,t}\in\cL(\HH_{p,0},\HH_{p,t})
\end{equation}
in the quantum bundle $\HH_p$ over $\R$ with respect to
$\nabla^{\HH_p}$.
Recall that $\tau^E_t:E_0\fl E_t$
has been defined by \cref{tauE}.
The following Theorem
shows that the parallel transport has the semi-classical
behaviour of a \emph{Toeplitz operator} as $p\fl+\infty$.

\begin{theorem}\label{Approx}
{\cite[Th.\,3.16]{Ioo18c}}
There exists a sequence
$\{\mu_{l,t}\in\CC^\infty(X,E_{p,t}\otimes E_{p,0}^*)\}_{l\in\N}$,
smooth in $t\in\R$, such that for all $k\in\N^*$,
there exists $C_k>0$ such that
\begin{equation}\label{genToepTau}
\Big\|\Tau_{p,t}-\sum_{l=0}^{k-1} p^{-l}P_{p,t}\mu_{l,t}P_{p,0}\Big\|_{p,0,t}\leq C_k p^{-k},
\end{equation}
for all $p\in\N$ and $t\in\R$.
Furthermore, there is a natural function $\mu_t\in\cinf(X,\C)$
such that the first coefficient $\mu_{0,t}$ satisfies
\begin{equation}\label{g0t}
\mu_{0,t}=\mu_t\tau^E_t\,.
\end{equation}
\end{theorem}

To describe the function $\mu_t\in\cinf(X,\C)$ of \cref{g0t},
let us describe the local setting involved in the proof of the
above theorem in \cite{Ioo18c}. For any $t\in\R$, using
the fact that the almost complex structures
$J_0\in\End(TX)$ and $J_t\in\End(TX)$ are both compatible
with the same symplectic form $\om$, we get a splitting
\begin{equation}
TX_{\C}=T^{(1,0)}X_0\oplus T^{(0,1)}X_t\,
\end{equation}
into the holomorphic subspace $T^{(1,0)}X_0$ of $TX_\C$ associated
to $J_0\in\End(TX)$ and the anti-holomorphic subspace
$T^{(0,1)}X_t$ of $TX_\C$ associated with $J_t\in\End(TX)$
as in \cref{splitc}.
We write
\begin{equation}\label{Pi0t}
\Pi_0^t\in\End(TX_\C)
\end{equation}
for the projection operator onto $T^{(1,0)}X_0$ with
kernel $T^{(0,1)}X_t$. In a dual way, we write
$\overline{\Pi}_t^0\in\End(TX_\C)$ for the projection operator onto $T^{(0,1)}X_t$ with
kernel $T^{(1,0)}X_0$. Considering its restriction to
$T^{(0,1)}X_0$ and via the isomorphism
$T^{(0,1)}X_t\simeq T^{(1,0)*}X_t$ induced by $g^{TX}_t$ for all
$t\in\R$,
it induces an isomorphism
\begin{equation}
\det(\overline{\Pi}_t^0):K_{X,0}\longrightarrow K_{X,t}\,
\end{equation}
of the respective canonical line bundles over $X$.
Recall the connection $\nabla^{K_X}$ on the
canonical line bundle $K_X$ over $\R\times X$
of \cref{KXdef}, inducing $\tau_t^{K_X}:K_{X,0}\fl K_{X,t}$
by \cref{tauE}.
Then by \cite[(5.4)]{Ioo18c},
the function $\mu_t\in\cinf(X,\C)$ of \cref{g0t}
satisfies
\begin{equation}\label{mubar}
\bar{\mu}_t^2(x)=\det(\overline{\Pi}_t^0)^{-1}\tau_t^{K_X}\,,
\end{equation}
for all $t\in\R$, via the canonical identification
$K_{X,0}\otimes K_{X,0}^*\simeq\C$.

The main tool of the proof of \cref{Approx} in \cite{Ioo18c}
is the local study of the \emph{Schwartz kernel} with respect
to $dv_X$ of the parallel transport operator. For any linear operator
$\K_{p,t}\in\cL(\HH_{p,0},\HH_{p,t})$, write
$\K_{p,t}(\cdot,\cdot)\in
\cinf(X\times X,E_{p,t}\boxtimes E_{p,0}^*)$ for the Schwartz
kernel with respect to $dv_X$ of the bounded operator
\begin{equation}\label{KptL2}
\K_{p,t}:=P_{p,t}\K_{p,t}P_{p,0}:L^2(X,E_{p,0})
\longrightarrow L^2(X,E_{p,t})\,,
\end{equation}
defined for any $s\in\cinf(X,E_{p,0})$
and $x\in X$ by the formula
\begin{equation}\label{kerdef}
\K_{p,t}s\,(x)=\int_X \K_{p,t}(x,y)s(y)\,dv_X(y)\,.
\end{equation}
The existence of a smooth Schwartz
kernel is an immediate consequence
of the fact that the image of \cref{KptL2} is finite dimensional.
In the case $t=0$, so that $\K_{p,0}\in\End(\HH_{p,0})$
and $\K_{p,0}(x,x)\in\End(E_{p,0})_x$ for all $x\in X$,
we have the following basic trace formula,
\begin{equation}\label{Trfla}
\Tr[\K_{p,0}]=\int_X
\Tr\left[\K_{p,0}(x,x)\right]\,dv_X(x)\,.
\end{equation}
Fix $\epsilon>0$ and consider a collection of diffeomorphisms
\begin{equation}\label{chart}
B^{T_{x_0}X}(0,\epsilon)\xrightarrow{~~\sim~~} V_{x_0}\subset X\,,
\end{equation}
varying smoothly with $x_0\in X$,
sending $0$ to $x_0\in X$ and with differential at $0$ inducing the
identity of $T_{x_0}X$.
%be smaller as the injectivity radius of
%$(X,g^{TX}_0)$, and fix $x_0\in X$.
%We write $B^X(x_0,\epsilon_0)$ for the geodesic ball with respect to
%$g^{TX}_0$ of center $x_0\in X$ and radius $\epsilon_0$.
%Using an orthonormal frame of $(TX,g^{TX}_0)$ around $x_0$,
%we identify $B^X(x_0,\epsilon_0)$ with the ball
%$B(0,\epsilon_0)$ of center $0$ and radius $\epsilon_0$ in
%$\R^{2n}$ via the exponential map of $(X,g^{TX})$,
%and call these coordinates the \emph{normal coordinates}.
For any $x_0\in X$ and $t\in\R$, we pullback $(L,h^L,\nabla^L)$ and
$(E_t,h^{E_t},\nabla^{E_t})$ over $V_{x_0}$ in this chart,
and identify them with their central fibre $L_{x_0}$ and $E_{t,x_0}$
by parallel transport along
radial lines of $B^{T_{x_0}X}(0,\epsilon)$. We then identify
$L_{x_0}$ with $\C$ by the choice of a unit vector.
%, and identify
%$E_{t,x_0}$ with $E_{0,x_0}$ by parallel transport by $\tau^E_t$
%in \cref{tauE} along horizontal lines of the tautological
%fibration \cref{tautfib}.
%let
%$\kappa_{x_0}\in\cinf(B(0,\epsilon_0,\R)$ be
%such that for any $Z\in B(0,\epsilon_0)$ in
%normal coordinates,
%\begin{equation}\label{defkappa}
%dv_X(Z)=\kappa_{x_0}(Z)dZ.
%\end{equation}
%Then $\kappa_{x_0}(0)=1$.
%write $|\cdot|$ for the canonical norm of $\R^{2n}$.
For any $\K_{p,t}\in\cL(\HH_{p,0},\HH_{p,t})$,
we write
$\K_{p,t,x_0}(\cdot,\cdot)$ for the image in this trivialization
of its Schwartz kernel over $V_{x_0}\times V_{x_0}$.
Then $\K_{p,t,x_0}(\cdot,\cdot)$ can be seen as
the evaluation at $x_0\in X$ of the pullback of $E_t\otimes E_0^*$
over the fibred product $B^{TX}(0,\epsilon)\times_X B^{TX}(0,\epsilon)$
over $X$, and for any $m\in\N$, let $|\cdot|_{\CC^m(X)}$
be a local $\CC^m$-norm on this bundle
induced by derivation with respect to $x_0\in X$.
%on $\cinf(TX\times_X TX,\C)$, induced by derivation
%in the direction of $X$ only via the Levi--Civita connection.
%Then $R_{x_0}$ can be seen as the evaluation at $x_0\in X$ of
%a function defined on an open set of the fibre product
%$TX\times_X TX$.

For any $Z,\,Z'\in T_{x_0}X$ and $t\in\R$,
we use the following local model
\begin{equation}\label{locmodT}
\cT_{t,x_0}(Z,Z'):=
\exp\big(-\pi\left[\left\langle\Pi_0^t(Z-Z'),(Z-Z')
\right\rangle+\i\om(Z,Z')\right]\big)\,,
\end{equation}
where $\<\cdot,\cdot\>$ is the scalar product on $T_{x_0}X$
induced by the metric $g^{TX}_0$ defined by \cref{gTXintro}.
For any $F_{t,x_0}(Z,Z')\in E_t\otimes E_0^*$  polynomial in
$Z,\,Z'\in T_{x_0}X$, we write
\begin{equation}\label{polnot}
F\cT_{t,x_0}(Z,Z'):=F_{t,x_0}(Z,Z')\cT_{t,x_0}(Z,Z')
\in E_t\otimes E_0^*\,.
\end{equation}
For any function $h\in\cinf(X,\R)$, we will use repeatedly
in the sequel the following form of Taylor expansion around
any $x_0\in X$ up to order $k-1\in\N$,
as $|Z|\fl 0$
in the chart \cref{chart} above,
\begin{equation}\label{Taylordef}
\begin{split}
h(Z) &=h(x_0)+
\sum_{r=1}^{k-1} \sum_{|\alpha|=r}
\frac{\partial^{r} h}{\partial Z^\alpha}\frac{Z^\alpha}{\alpha!}
+O(|Z|^k)\\
& =h(x_0)+
\sum_{r=1}^{k-1} p^{-r/2} \sum_{|\alpha|=r}
\frac{\partial^{r} h}
{\partial Z^\alpha}
\frac{(\sqrt{p}Z)^\alpha}{\alpha!}
+p^{-\frac{k}{2}}O(|\sqrt{p}Z|^{k})\,.
\end{split}
\end{equation}
Write $d^X(\cdot,\cdot)$ for the Riemannian distance of
$(X,g^{TX}_0)$, and for any $m\in\N$,
let $|\cdot|_{\CC^m}$ be the local $\CC^m$-norm induced by derivation
with respect to $\nabla^{E_p}$ over $\R\times X$.
Then \cref{Approx} is based on the following result.

\begin{theorem}\label{BTasy}
{\cite[Th.\,4.3]{Ioo18c}}
Consider a collection of charts of the form \cref{chart},
varying smoothly with $x_0\in X$, sending $0$ to $x_0$ and
with  differential at $0$ inducing the identity of $T_{x_0}X$.
Then for any $m,\,k\in\N,\,\theta\in\,]0,1[$
and any compact subset $K\subset\R$,
there is
$C_k>0$ such that for all $p\in\N$ and $t\in K$,
we have
\begin{equation}\label{thetafla}
\left|\Tau_{p,t}(x,y)\right|_{\CC^m}\leq C_kp^{-k}\,
~~\text{as soon as}~~d^X(x,y)>\epsilon
p^{-\frac{\theta}{2}}\,.
\end{equation}
Furthermore, there is a family
$\{G_{r,t,x_0}(Z,Z')\in E_{t,x_0}\otimes E_{0,x_0}^*\}_{r\in\N}$
of polynomials
in $Z,Z'\in T_{x_0}X$ of the same parity as $r$,
depending smoothly on $x_0\in X$, such that
for any $m\,,m',\,l,\,k\in\N,\,\delta\in]0,1[$ and any
compact subset $K\in\R$,
there is $C>0$ and $\theta\in\,]0,1[$ such that for any
$x_0\in X, p\in\N$ and $Z,\,Z'\in T_{x_0}X$ with
$\,|Z|,|Z'|<\epsilon p^{-\theta/2}$, we have
\begin{multline}\label{expTau}
\sup_{|\alpha|+|\alpha'|=m}
\Big|\Dk{l}{t}\Dk{\alpha}{Z}\Dk{\alpha'}{Z'}\big(p^{-n}
\Tau_{p,t,x_0}(Z,Z')\\
-\sum_{r=0}^{k-1} p^{-r/2}G_r\cT_{t,x_0}(\sqrt{p}Z,\sqrt{p}Z')\big)\Big|_{\CC^{m'}(X)}
\leq Cp^{-\frac{k-m}{2}+\delta}\,,
\end{multline}
where $G_{0,t,x_0}(Z,Z')$ is constant in $Z,\,Z'\in T_{x_0}X$,
given by
\begin{equation}\label{|J|0}
G_{0,t,x_0}(Z,Z')=\bar{\mu}_t^{-1}(x_0)\tau^E_{t,x_0}\,.
\end{equation}
\end{theorem}

Let us now consider a diffeomorphism $\varphi:X\fl X$ preserving
the symplectic form $\om$, together with a lift $\varphi^L:L\fl L$
to the total space of $L$ preserving metric and connection,
and assume that for some $t_0\in\R$, we have
\begin{equation}
J_{t_0}=d\varphi\,J_0\,d\varphi^{-1}\,.
\end{equation}
Assume further that there is a lift $\varphi^E:E_0\fl E_{t_0}$
of $\varphi$ preserving metric and connection,
and write $\varphi_{p}:E_{p,0}\fl E_{p,t_0}$
for the induced lift on $E_p$, for all $p\in\N$.
For any $s\in\cinf(X,E_{p,{t_0}})$, we define the \emph{pullback}
$\varphi_p^*s\in\cinf(X,E_{p,0})$ by the formula
\begin{equation}\label{pullbackdef}
(\varphi_p^*s)(x):=\varphi_{p}^{-1}s(\varphi(x))\,.
\end{equation}
This induces by restriction a unitary isomorphism
\begin{equation}
\varphi_p^*:\HH_{p,t_0}\xrightarrow{~~\sim~~}\HH_{p,0}\,.
\end{equation}
Then \cref{sctrth} gives a semi-classical estimate
for the trace of the endomorphism
$\varphi^{*}_{p}\Tau_{p,{t_0}}\in\End(\HH_{p,0})$
as $p\fl+\infty$, using the trace formula \cref{Trfla}.
To this end, we need the following assumption.

\begin{defi}\label{nondegdef}
The fixed point set $X^\varphi\subset X$ of a diffeomorphism
$\varphi:X\fl X$
is said to be \emph{non-degenerate} over an open set $U\subset X$
if $X^\varphi\cap U$ is a
proper submanifold of $\overline{U}$ satisfying
\begin{equation}\label{nondegfla}
T_xX^\varphi=\Ker(\Id_{T_xX}-d\varphi_x)~~\text{for all}~~x\in
X^\varphi\cap U.
\end{equation}
\end{defi}

As the lift $\varphi^L$
preserves $h^L$ and $\nabla^L$,
its value $\beta\in\C$ via the canonical identification
$L\otimes L^*\simeq\C$
is locally constant over $X^\varphi$, and satisfies $|\beta|=1$.

Let $Y\subset X$ be a submanifold, and let $N$ be a subbundle of
$TX$ over $Y$ transverse to $TY$. Write $g^N$ for the Euclidean
metric on $N$ induced by the metric $g^{TX}_0$
defined by \cref{gTXintro}, and write
$|dv|_{TX}$ and $|dv|_{N}$ for the Riemannian densities of
$(TX,g^{TX}_0)$ and $(N,g^N)$. We denote by $|dv|_{TX/N}$ the density
over $Y$ defined by the formula
\begin{equation}
|dv|_{TX}=|dv|_{TX/N}|dv|_{N}\,.
\end{equation}
We write $P^N:TX\fl N$ for the
orthogonal projection with respect to $g^{TX}_0$
of vector bundles over $Y$.

\begin{theorem}\label{sctrth}
{\cite[Th.\,1.2]{Ioo18c}}
Assume that the fixed point set $X^\varphi$ of $\varphi:X\fl X$
is non-degenerate over $X$, and write $\{X_j\}_{1\leq j \leq m}$ 
for the set of its connected components. Then there exist
densities $\nu_{r}$ over $X^\varphi$ for any $r\in\N$
such that for any $k\in\N^*$ and as $p\fl +\infty$,
\begin{equation}\label{indeqfle}
\Tr[\varphi^{*}_{p}\Tau_{p,{t_0}}]
=\sum_{j=1}^m p^{\dim X_j/2}e^{-2\pi\sqrt{-1}p\lambda_j}
\left(\sum_{r=0}^{k-1} p^{-r} \int_{X_j}\nu_{r}
+O(p^{-k})\right)\,,
\end{equation}
where $e^{2\pi\i \lambda_j}$ is the constant value of
$\varphi^L$ over $X_j$, for some $\lambda_j\in\R$. Furthermore, 
for any $x\in X^\varphi$ we have
\begin{multline}\label{nu0}
\frac{\nu_{0}}{|dv|_{TX/N}}(x)
=\Tr_{E_x}
[\varphi^{E,-1}\tau^E_{{t_0}}]
\left(\det(\overline{\Pi}_{t_0}^0)^{-1}
\tau_{t_0}^{K_X}\right)^{-\frac{1}{2}}_x
\int_{N_x}\cT_{{t_0},x}(d\varphi.Z,Z)dZ\\
=\Tr_{E_x}[\varphi^{E,-1}\tau^E_{t_0}]
\left(\det(\overline{\Pi}_{t_0}^0)^{-1}
\tau_{t_0}^{K_X}\right)^{-\frac{1}{2}}_x
\\
\det{}^{-\frac{1}{2}}_{N_x}
\left[P^N(\Pi_{0}^{t_0}-d\varphi^{-1}\overline{\Pi}_{t_0}^0)
(\Id_{TX}-d\varphi)P^N\right]\,,
\end{multline}
for some natural choices of square roots,
where $N$ is any subbundle of $TX$ over $X^\varphi$ transverse to
$TX^\varphi$.
\end{theorem}

The first coefficient \cref{nu0} acquires a geometric interpretation
in the special case when the bundles $TX^\varphi$ and $N$ are
both preserved by $\varphi$ and $J_0$. In order to describe it,
let $\varphi^{K_X}:K_{X,0}\longrightarrow K_{X,{t_0}}$ be
the natural action
induced by $\varphi$, and recall that
$\tau^{K_X}_{t_0}:K_{X,0}\longrightarrow K_{X,{t_0}}$ has been
defined in \cref{tauE}.
Then $\varphi^{K_X,-1}\tau^{K_X}_{t_0}\in\cinf(X,\C)$ via the canonical
identification $K_{X,0}\otimes K_{X,0}^*\simeq\C$, and
one can compute the following.

\begin{prop}\label{loctrcT}
{\cite[Lemma 5.1]{Ioo18c}}
Assume that the fixed point set $X^{\varphi}$ of $\varphi:X\fl X$
is non-degenerate over $X$, and that there exists a subbundle $N$
of $TX$ over $X^\varphi$ transverse to $TX^\varphi$
such that $TX^\varphi$ and $N$ are
both preserved by $\varphi$ and $J_0$. Then
we have the following formula, for all $x\in X^\varphi$,
\begin{multline}\label{loctrcTfla}
\left(\det(\overline{\Pi}_{t_0}^0)^{-1}
\tau_{t_0}^{K_X}\right)^{-\frac{1}{2}}_x
\int_{N_x}\cT_{{t_0},x}(d\varphi.Z,Z)dZ\\
=(-1)^{\frac{\dim N_x}{4}}
(\varphi^{K_X,-1}\tau^{K_X}_{{t_0}})^{-\frac{1}{2}}_x\,
|\det{}_{N_x}(\Id_{N}-d\varphi|_{N})|^{-\frac{1}{2}}\,,
\end{multline}
for some natural choices of square roots.
\end{prop}
The previous result acquires an even cleaner formulation in the
case when $(E,h^E,\nabla^E)$ satisfies
\begin{equation}\label{metafla}
E^2=K_X\,,
\end{equation}
as a line bundle over $\R\times X$
with induced metric and connection. Such a line bundle exists
if and only if the first Chern class $c_1(TX)\in H^2(X,\Z)$
of $TX$ is even, and the choice of a complex line bundle $E$
satisfying \cref{metafla}
is called a \emph{metaplectic structure} on $X$.
We write $E=:K_X^{1/2}$, and
call it the \emph{metaplectic correction}. We then get the following
straightforward corollary of \cref{loctrcT}.

\begin{cor}\label{loctrcTmeta}
Consider the assumptions of \cref{loctrcT}, and assume
further that $X$ admits a metaplectic structure.
Then if $E=K_X^{1/2}$ is the associated metaplectic correction
over $\R\times X$,
the first coefficient $\nu_{0}$ of
\cref{nu0} satisfies the formula
\begin{equation}
\nu_0=(-1)^{\frac{\dim N}{4}}\,
|\det{}_N(\Id_{N}-d\varphi|_N)|^{-\frac{1}{2}}\,|dv|_{TX/N}\,,
\end{equation}
for some natural choices of square roots.
\end{cor}

In the sequel, we will write $|\cdot|_p$ for the norm induced
on $E_p\otimes E_p^*$ by $h^{E_p}$, for all
$p\in\N$.

%In particular,
%we prove in \cref{gutzwiller} a corresponding Gutzwiller type formula,
%extending the results of \cite[Th.4.2]{BPU98}.

%
%Let $(X,\om)$ be a compact symplectic manifold, and let $J\in\End(TX)$ be an almost complex structure on $X$ compatible with $\om$. Given a smooth function $f\in\cinf(X,\R)$, the associated Hamiltonian flow $\varphi_t:X\rightarrow X,\,t\in\R$, does not preserve $J$ in general, but instead induces an automorphism over a tautological quantized fibration $\pi:X\times\R\fl\R$, with almost complex structure $J_t=d\varphi_t J(d\varphi_t)^{-1}$ over $X_t$ for all $t\in\R$. As the Berezin--Toeplitz quantization of a compact symplectic manifold depends on the choice of an almost complex structure $J$, in order to study the quantization of Hamiltonian dynamics associated with the function $f$, it is then natural to the study the quantization of the induced quantized automorphism.

\section{Quantum evolution operators}\label{quantevosec}

%In this section, we use the parallel transport of \cref{pt} to study the Berezin--Toeplitz quantization of Hamiltonian flows, identifying it with the quantum evolution operator associated with the given Hamiltonian dynamics. 
Let $(X,\om)$ be a compact symplectic manifold without boundary
endowed with
$(L,h^L,\nabla^L)$ satisfying the prequantization
condition \cref{preq}, and consider a smooth function $f\in\cinf(X,\R)$. 
The \emph{Hamiltonian vector field} $\xi_f\in\cinf(X,TX)$ associated
to $f$ is defined by the formula
\begin{equation}\label{hamiltonian}
\iota_{\xi_f}\om=df\,.
\end{equation}
The \emph{Hamiltonian flow} of $f$ is the flow of diffeomorphisms
$\varphi_t:X\rightarrow X$ defined for all $t\in\R$ by
\begin{equation}\label{Hamflow}
\left\{
\begin{array}{l}
  \dt\varphi_t=\xi_{f}\,, \\
  \\
 \varphi_0 = \Id_X\,.
\end{array}
\right.
\end{equation}
By definition \cref{hamiltonian} of a Hamiltonian vector field
and by Cartan formula, the Hamiltonian flow $\varphi_t:X\fl X$
preserves $\om$ for all $t\in\R$.
Let
$\til{\xi}_f\in\cinf(L,TL)$ be the horizontal lift of
$\xi_f$ to the total space of $L$ with respect to
$\nabla^L$, and let $\bold{t}\in\cinf(L,TL)$ be the canonical vector field on the total space of $L$ defined by
\begin{equation}
\bold{t}=\dt\Big|_{t=0}e^{2\pi\sqrt{-1}t}\,,
\end{equation}
for the action of $e^{2\pi\sqrt{-1}t}$ by complex multiplication in the
fibres. Then the flow \cref{Hamflow} lifts to a flow
$\varphi_t^L:L\fl L$
on the total space of $L$ over $X$, defined for all $t\in\R$
by the formula
\begin{equation}\label{liftflow}
\left\{
\begin{array}{l}
  \dt\varphi_t^L=\til{\xi}_{f}+f\bold{t}\,, \\
  \\
 \varphi_0^L = \Id_L\,.
\end{array}
\right.
\end{equation}
Note that both $\varphi_t$ and $\varphi_t^L$ define $1$-parameter
groups, as both vector fields defining them do not depend on
$t\in\R$.
From the definition \cref{hamiltonian} of the Hamiltonian vector field
of $f$, we see that $\varphi_t^L$ is the unique lift of
$\varphi_t$ to $L$ preserving the connection $\nabla^L$, for all
$t\in\R$.
More specifically, for any $t\in\R$, recall the pullback of
$s\in\cinf(X,L)$ by $\varphi_t$ defined by formula \cref{pullbackdef}.
Then for any vector field $v\in\cinf(X,TX)$, we get from
\cref{hamiltonian} and \cref{liftflow} that
\begin{equation}\label{presconn}
\varphi_t^*\nabla_v^Ls=\nabla_{d\varphi_t.v}^L\varphi_t^*s\,.
\end{equation}
On the other hand, we also deduce from \cref{liftflow}
the following \emph{Kostant formula}, for all $t\in\R$,
\begin{equation}\label{Kostantfla}
\dt\varphi^{*}_ts=
\left(\nabla^L_{\xi_{f}}-2\pi\sqrt{-1} f\right)
\varphi^{*}_ts\,.
\end{equation}
For any $p\in\N$ and $t\in\R$, let us write $\varphi_{t,p}$ for the
flow induced by $\varphi_t^L$ on the total space of $L^p$,
and $\varphi_{t,p}^*$ for the associated pullback as in
\cref{pullbackdef}. Then for any $s\in\cinf(X,L^p)$, the Kostant formula
\cref{Kostantfla} becomes
\begin{equation}\label{Kostantflap}
\dt\varphi^*_{t,p}s=
\left(\nabla^{L^p}_{\xi_{f}}-2\pi\sqrt{-1} p f\right)
\varphi^{*}_{t,p}s\,.
\end{equation}
This formula
characterizes $f\in\cinf(X,\R)$ as the \emph{Kostant moment map}
for the action of $\R$ on $(L^p,h^{L^p},\nabla^{L^p})$ induced by
$\varphi_t$ for all $t\in\R$.

Note that \cref{hamiltonian} implies that $f(\varphi_t(x))=f(x)$
for all $t\in\R$ and $x\in X$, and \cref{Hamflow} implies that
$d\varphi_t.\xi_f=\xi_f$ for all $t\in\R$.
%Using \cref{presconn},
%we then see that
%\begin{equation}
%\varphi^{*}_{t,p}\left(\nabla^{L^p}_{\xi_{f}}-2\pi\sqrt{-1} p f\right)
%=\left(\nabla^{L^p}_{\xi_{f}}-2\pi\sqrt{-1} p f\right)
%\varphi^{*}_{t,p}\,,
%\end{equation}
%where both sides are seen as operators acting on $\cinf(X,L^p)$.
For all $x\in X$, we write
\begin{equation}\label{tautp}
\tau_{t,p}:L_x^p\longrightarrow L_{\varphi_t(x)}^p
\end{equation}
for the parallel transport
along the path $s\mapsto\varphi_s(x)$ for $s\in[0,t]$.
We can then reformulate the Kostant formula \cref{Kostantflap} as
\begin{equation}\label{alphaf}
\varphi_{t,p}^{-1}\tau_{t,p}=e^{-2\pi\i tp f}\in\cinf(X,\C)\,,
\end{equation}
via the canonical identification $L\otimes L^*\simeq\C$.

Let us now consider an almost complex structure $J_0\in\End(TX)$ over
$X$ compatible with $\om$. Then for any $t\in\R$, the formula
\begin{equation}\label{J_t}
J_t:=d\varphi_t J_0\,d\varphi_{t}^{-1}\in\End(TX)
\end{equation}
defines a path $\{J_t\in\End(TX)\}_{t\in\R}$ of almost complex
structures over $X$ compatible with $\om$.
For any $t\in\R$, we write $\HH_{p,t}$ for the space of almost
holomorphic sections with respect to $J_t$ defined in
\cref{specdeltapphi}. Then for any $t_0\in\R$ and $p\in\N$,
the pullback \cref{phitcinf}
induces by restriction a bijective linear map
\begin{equation}\label{isoHptHp0}
\varphi^{*}_{t_0,p}:\HH_{p,t+t_0}\longrightarrow\HH_{p,t}\,,
\end{equation}
for all $t\in\R$. For any fixed $p\in\N$, this implies
in particular that the dimension $\dim\HH_{p,t}$ does not depend on
$t\in\R$, so that the quantum bundle
$(\HH_p,h^{\HH_p},\nabla^{\HH_p})$ of \cref{quantbdledef}
is well defined over $\R$ for all $p\in\N$.
%For any $t\in\R$, let now $\tilde{\Phi}_t$ be the quantized automorphism induced on the universal family over $B$ by $\til{\varphi}^f_t$, and recall that for any $p\in\N$, the induced map $\tilde{\Phi}^*_{t,p}$ on the infinite dimensional bundle of fibrewise smooth sections of $L_p$ preserves the subbundle $\HH_p$. Recall the discussion on flat fibrations in \cref{sectiontrace}. The horizontal vector field $\delt^H$ on $M$ is the lift of the canonical vector field $\delt$ on $\R=B$ via the decomposition $M=X\times\R$, and $L\simeq L\times\R$ is trivial in the horizontal direction as a Hermitian bundle with connection.
%
%Write $\delt$ for the canonical vector field of $\R$, and
Recall the tautological fibration $\pi:\R\times X\fl X$
considered in \cref{tautfib}
together with all the data induced by $\{J_t\in\End(TX)\}_{t\in\R}$,
and consider the flow
over $\R\times X$ defined for all $t_0\in\R$ by
\begin{equation}\label{Phitdef}
\begin{split}
\Phi_{t_0}:\R\times X &\longrightarrow\R\times X\\
(t,x)&\longmapsto (t+t_0,\varphi_{t_0}(x))\,.
\end{split}
\end{equation}
For any $p\in\N$ and $t_0\in\R$, the lift $\varphi_{t_0,p}$ of
$\varphi_{t_0}$ to $(L^p,h^{L^p},\nabla^{L^p})$ over $X$ induces
tautologically a lift $\Phi^*_{t_0,p}$ of $\Phi_{t_0}$
to the pullback of $(L^p,h^{L^p},\nabla^{L^p})$ over $\R\times X$
via the second projection.
For any section $s\in\cinf(\R\times X,L^p)$ over $\R\times X$
and any $t\in\R$,
write $s_{t}\in\cinf(X,L^p)$ for the section over $X$ defined
by $s_{t}(x):=s(t,x)$ for all $x\in X$.
Then for any $t_0\in\R$ and $p\in\N$,
the pullback of $s\in\cinf(\R\times X,L^p)$
by $\Phi_{t_0}$ is given for any $t\in\R$ by the formula
\begin{equation}\label{Phitp*}
(\Phi_{t_0,p}^{*}s)_{t}=\varphi_{t_0,p}^*s_{t+t_0}\,.
\end{equation}
By \cref{isoHptHp0}, the pullback $\Phi_{t_0,p}^{*}$
preserves the smooth sections $\cinf(\R,\HH_p)$
of the quantum bundle, seen as a subspace of $\cinf(\R\times X,L^p)$
as in \cref{quantbdledef}.

We still write $\xi_f$ for the pullback of the Hamiltonian vector field
$\xi_f\in\cinf(X,TX)$ to a vertical vector field
over $\pi:\R\times X\fl\R$ via the second projection, and write
$\delt$ for the horizontal vector field over $\pi:\R\times X\fl\R$
induced by the canonical vector field of $\R$.
By definition of the pullback of $(L,h^L,\nabla^L)$ to
$\R\times X$, for any $s\in\cinf(\R\times X,L^p)$ and $t\in\R$ we have
%\cref{setting} that in the bundle of smooth fibrewise sections of a tautological fibration, there is a natural identification $\cinf(M,L^p)=\cinf(\R,\cinf(X,L^p))$ for any $p\in\N$, through which for any $s\in\cinf(M,L^p)$ and all $t\in\R$, we get
\begin{equation}\label{ident}
\left(\nabla^{L^p}_{\delt} s\right)_t=\dt s_t\,.
\end{equation}
Recall \cref{quantbdledef} for the connection $\nabla^{\HH_p}$,
and note that for all $t,\,t_0\in\R$, we have
\begin{equation}\label{PhitPi=PiPhit}
\Phi^*_{t_0,p}P_{p,t}=P_{p,t+t_0}\Phi^*_{t_0,p}\,.
\end{equation}
%preserves the $L^2$-connection $\nabla^{\HH_p}$ of
%\cref{quantbdledef} by \cref{presconn}.
%$\cinf(M,L^p)\fl\cinf(\R,\HH_p)$ on the fibrewise almost holomorphic sections for all $p\in\N,\,t\in\R$. For any section $s\in\cinf(\R,\HH_p)\subset\cinf(M,L^p)$ and all $p\in\N,~t\in\R$, we get by definition that the pullback $\til\Phi^*_{t,p}s$ belongs to $\cinf(\R,\HH_p)\subset\cinf(M,L^p)$ as well, and so that in particular 
%\begin{equation}
%P_p\dt\til\Phi^*_{t,p}s=\dt\til\Phi^*_{t,p}s.
%\end{equation}
%Furthermore, again by definition we have that $\til\Phi^*_{t,p}$ commutes with $P_p$. 
Using
\cref{isoHptHp0,Phitdef,Phitp*,ident,PhitPi=PiPhit}, for any
smooth section $s\in\cinf(\R,\HH_p)\subset\cinf(\R\times X,L^p)$
and seeing the orthogonal projection
$P_p:\cinf(\R\times X,L^p)\fl\cinf(\R,\HH_p)$
as a global endomorphism, we get the following
quantized version of the Kostant formula \cref{Kostantflap},
for all $t_0\in\R$,
\begin{equation}\label{KostantflaHp}
\begin{split}
\dt\Big|_{t=t_0}\Phi^{*}_{t,p}s & =
\left(\nabla^{L^p}_{\xi_f}-2\pi\sqrt{-1} pf\right)\Phi^{*}_{t_0,p}s+
\nabla^{L^p}_{\delt}\Phi^{*}_{t_0,p}s\\
& =P_p
\left(\nabla^{L^p}_{\xi_f+\delt}-2\pi\sqrt{-1} pf\right)
\Phi^{*}_{t_0,p}P_p s\\
& =\left(P_p\nabla^{L^p}_{\delt}P_p
+P_p(\nabla^{L^p}_{\xi_f}-2\pi\sqrt{-1} pf)P_p\right)\Phi^{*}_{t_0,p}s\\
& =\left(\nabla^{\HH_p}_{\delt}-2\pi\sqrt{-1} p
P_p\left(f+\frac{\sqrt{-1}}{2\pi p}
\nabla^{L^p}_{\xi_f}\right)P_p\right)\Phi^{*}_{t_0,p}s\,.
\end{split}
\end{equation}
%Note that we used the tautological fact that $P_p$ acts as the identity
%on $\cinf(\R,\HH_p)$, together with the fact that
%$\Phi_{t,p}^{*}$ preserves $\cinf(\R,\HH_p)$ inside
%$\cinf(\R\times X,L^p)$ for all $t\in\R$.
Seeing the path $\{J_t\in\End(TX)\}_{t\in\R}$ as lying in
the space $\JJ_\om$ of almost complex structures compatible with $\om$
and comparing with the usual Kostant formula \cref{Kostantflap},
we can interpret the \emph{quantized Kostant formula}
\cref{KostantflaHp}
by stating that the Kostant--Souriau operator
$Q_p(f)\in\End(\HH_p)$ given by formula \cref{quanthamflaintro}
induces a moment map on the
quantum bundle $\HH_p$ over $\JJ_\om$
for the natural action of the group of Hamiltonian diffeomorphisms
$\textup{Ham}(X,\om)$ on $\JJ_\om$ defined by
\cref{J_t}.

The relevance of the Kostant--Souriau operator
in Kähler geometry goes back
to the work of Cahen, Gutt and Rawnsley \cite{CGR90} relating it to
Berezin's quantization of Kähler manifolds \cite{Ber74},
and the moment map picture described above
has been introduced by Donaldson in \cite{Don01}.
As explained for example in \cite[\S 1]{FU07},
we can consider the line bundle $\det(\HH_p)$ over any compact
submanifold of $\JJ_\om$ for $p\in\N$ big enough,
and the curvature of the connection
$\nabla^{\det\HH_p}$ induced by the $L^2$-connection
\cref{connectionL2}
on $\det(\HH_p)$ defines a natural symplectic form via the
prequantization formula \cref{preq}. In the spin$^c$ Dirac
operator case, which implies the Kähler case, it follows from
the asymptotics of the curvature of $\nabla^{\HH_p}$ as
$p\fl+\infty$
established by Ma and Zhang in \cite[Th.\,2.1]{MZ07}.
Then the quantized Kostant formula \cref{KostantflaHp} shows that
the Hamiltonian flow associated with the function
$\det(Q_p(f))$ on $\JJ_\om$
is precisely the action of the Hamiltonian flow of $f$ on
$\JJ_\om$ defined by \cref{J_t}.

On the other hand, as described for example in
\cite[\S\,9.7]{Woo92},
the quantum dynamics is given by the $1$-parameter family
of unitary operators
generated by the quantum Hamiltonian operator
acting on a \emph{fixed} space of quantum states.
For any $p\in\N$, we thus consider the quantum Hamiltonian
operator $Q_p(f)$ restricted to the space $\HH_{p,0}$ of almost
holomorphic sections with respect to our initial almost
complex structure $J_0$.
This induces a one-parameter family
$\exp\left(2\pi\sqrt{-1}tpQ_p(f)\right)\in\End(\HH_{p,0})$
of unitary operators defined for all $t\in\R$ by
\begin{equation}\label{quantevodef}
\left\{
\begin{array}{l}
  \dt\exp\left(2\pi\sqrt{-1}tpQ_p(f)\right)
  =2\pi\sqrt{-1}pQ_p(f)\exp\left(2\pi\sqrt{-1}tpQ_p(f)\right)\,, \\
  \\
 \exp\left(2\pi\sqrt{-1}tpQ_p(f)\right)\big|_{t=0} = \Id_{\HH_{p,0}}\,.
\end{array}
\right.
\end{equation}
Writing $\Tau_{p,t}:\HH_{p,0}\fl\HH_{p,t}$ for the parallel transport
with respect to $\nabla^{\HH_p}$ over $\R$ as in \cref{setting},
the following Lemma establishes formula \cref{evointro}.

\begin{lem}\label{quantevolem}
For any $p\in\N$ and all $t\in\R$, we have the following 
equality
\begin{equation}\label{quantevo}
\exp\left(-2\pi\sqrt{-1}tpQ_p(f)\right)=\varphi^{*}_{t,p}
\Tau_{p,t}\in\End(\HH_{p,0})\,.
\end{equation}
\end{lem}
\begin{proof}
By \cref{quantbdledef} of the $L^2$-connection $\nabla^{\HH_p}$,
using \cref{ident} and the fact that $\Phi_{t_0,p}^*$ commutes with
$P_p$ when acting
on $\cinf(\R\times X,L^p)$ for any $t_0\in\R$ and $p\in\N$,
we have
\begin{equation}
\Phi_{t_0,p}^*\nabla^{\HH_p}_{\delt}
=\nabla^{\HH_p}_{\delt}\Phi_{t_0,p}^*\,.
\end{equation}
Furthermore, as $\varphi_{t_0,p}^*$ commutes with
$\left(\nabla^{L^p}_{\xi_{f}}-2\pi\sqrt{-1} p f\right)$
when acting on $\cinf(X,L^p)$,
by \cref{quanthamflaintro}
and the pullback formula \cref{Phitp*}, we get
\begin{equation}
\Phi_{t_0,p}^*Q_p(f)=Q_p(f)\Phi_{t_0,p}^*\,.
\end{equation}
This implies the following analogue of \cref{alphaf} for the
quantized Kostant formula \cref{KostantflaHp}, for all $t\in\R$,
\begin{equation}\label{deftautphit}
\dt\varphi_{t,p}^{*}\Tau_{p,t}=-2\pi\sqrt{-1} p Q_p(f)\varphi_{t,p}^*\Tau_{p,t}\,,
\end{equation}
which follows from the quantized Kostant formula
\cref{KostantflaHp} in the same way as \cref{alphaf} follows
from the usual Kostant formula \cref{Kostantflap}.
This proves the lemma.
\end{proof}

Before giving the applications of the results described in
\cref{setting} to the quantum evolution
operator defined above, let us illustrates its behaviour via the
following definition.

\begin{defi}\label{cohstatedef}
For any $x_0\in X$ and any unit vector $\zeta\in L_{x_0}$,
the associated \emph{coherent state} is the sequence
$\{s_{x_0,p}\in\HH_{p,0}\}_{p\in\N}$ defined for all $x\in X$ by
\begin{equation}\label{cohstate}
s_{x_0,p}(x)=P_{p,0}(x,x_0)\zeta^p\,,
\end{equation}
where $P_{p,0}(\cdot,\cdot)\in
\cinf(X\times X,L^p\boxtimes (L^p)^*)$ is the Schwartz kernel
with respect to $dv_X$ of the orthogonal projection operator
$P_{p,0}:\cinf(X,L^p)\fl\HH_{p,0}$.
\end{defi} 

Coherent states represent the quantization of a classical
particle located at $x_0\in X$ in phase space.
As one can readily check from the
definition, it is characterized by the fact
that its orthogonal in $\HH_{p,0}$ consists of sections vanishing
at $x_0\in X$. As shown in 
\cite[Th.\,0.1,\,\S\,1.1]{MM08a},
the sequence $\{s_{x_0,p}\}_{p\in\N}$
decreases rapidly as $p\fl+\infty$ outside any open
set containing $x_0$, while $|s_{x_0,p}(x_0)|_{L^p}$ is of order $p^n$.
Now using the following tautological identity of operators
acting on $\cinf(X,L^p)$,
\begin{equation}
\exp\left(2\pi\sqrt{-1}tpQ_p(f)\right)=
\exp\left(2\pi\sqrt{-1}tpQ_p(f)\right)P_{p,0}\,,
\end{equation}
and by the classical formula for the Schwartz kernel of the composition
of two operators, for any $t\in\R$ and $x\in X$, we get from
\cref{cohstatedef},
\begin{equation}\label{kerunit}
\begin{split}
\exp\left(2\pi\sqrt{-1}tpQ_p(f)\right)&s_{x_0,p}(x)\\
&=\int_X\exp\left(2\pi\sqrt{-1}tpQ_p(f)\right)(x,w)
P_{p,0}(w,x_0)\zeta^p\,dv_X(w)\\
&=\exp\left(2\pi\sqrt{-1}tpQ_p(f)\right)(x,x_0)\zeta^p\,.
\end{split}
\end{equation}
This shows that
the last line of \cref{kerunit}, seen
as a section of $L^p$ with respect to the variable $x\in X$,
can be interpreted as the quantum evolution at time $t\in\R$ of
the quantization of a classical particle at $x_0\in X$.
In particular, we expect this section to decrease rapidly
as $p\fl+\infty$ outside any open set containing the classical
evolution $\varphi_t(x_0)\in X$,
while its value at $\varphi_t(x_0)$ should be of order $p^n$.
The following result shows that this is indeed the case.

\begin{prop}\label{Utth}
For any $\epsilon>0$, $k,m\in\N,\,\theta\in\,]0,1[$
and any compact subset $K\subset\R$,
there exists $C_k>0$ such that
\begin{equation}\label{thetacongUt}
\left|\exp\left(2\pi\sqrt{-1}tpQ_p(f)\right)(x,y)\right|_{\CC^m}
\leq C_kp^{-k}
~~\text{as soon as}~~d^X(x,\varphi_t(y))>\epsilon
p^{-\frac{\theta}{2}}\,,
\end{equation}
for all $t\in K$.
Furthermore, there exist $a_r(t,x)\in\C$ for any $r\in\N$, depending
smoothly on $x\in X$ and $t\in\R$, 
such that for any $k\in\N^*$,
\begin{equation}\label{exppsiquant1Ut}
\exp\left(2\pi\sqrt{-1}tpQ_p(f)\right)(\varphi_t(x),x)
=p^{n} e^{2\pi\i tp f(x)}\left(\sum_{r=0}^{k-1}
p^{-r} a_{r}(t,x) + O(p^{-k})\right)\tau_{t,p}\,,
\end{equation}
with first coefficient $a_0$ satisfying the formula
\begin{equation}\label{exppsiquant1coeffUt}
a_0(t,x)^2=\left(\det(\overline{\Pi}_{-t}^0)^{-1}\tau_{-t}^{K_X}\right)_x^{-1}\,.
\end{equation}
In particular, it satisfies $a_0(t,x)\neq 0$ for all $t\in\R$ and
$x\in X$.
\end{prop}
\begin{proof}
Recall that for any $x,\,y\in X$ and $t\in\R$,
we get from \cref{quantevolem} and formula
\cref{pullbackdef} that
\begin{equation}\label{kerpsiUt}
\exp\left(-2\pi\sqrt{-1}tpQ_p(f)\right)(x,y)=
\varphi_{t,p}^{-1}\Tau_{p,t}
(\varphi_t(x),y)\,.
\end{equation}
Then \cref{thetacongUt} is a consequence of
the Kostant formula \cref{Kostantfla}, together with
the rapid decrease of $\Tau_{p,t}(\cdot,\cdot)$
outside of the diagonal
given by \cref{thetafla}.
%we get for any $l,\,m\in\N$ a constant $C_{m,l}>0$ such that for any
%$\theta\in\,[0,1[$ and $p\in\N$
%\begin{equation}\label{estdttau}
%|\Dk{m}{t}\varphi_{p,t}^{-1}\Tau_{p,t}
%(\varphi_t(x),y)|\leq C_{m,l}\sum_{r=0}^m p^r
%|\varphi_{p,t}^{-1}(\nabla^L_\xi)^{m-r}\Tau_{p,t}
%(\varphi_t(x),y)|\,,
%\end{equation}

Using the
exponentiation  \cref{alphaf} of Kostant formula,
rewrite \cref{kerpsiUt} as
\begin{equation}
\begin{split}
\exp\left(2\pi\sqrt{-1}tpQ_p(f)\right)(\varphi_t(x),x)&=
\varphi_{t,p}\Tau_{p,-t}
(\varphi_{-t}(\varphi_t(x)),x)\\
&=e^{2\i\pi tp f(x)}\tau_{t,p}
\Tau_{p,-t}(x,x)\,.
\end{split}
\end{equation}
We can thus apply \cref{BTasy} with $x_0=x$
for $Z=Z'=0$, and
noting that $J_{2q+1}(0,0)=0$ for all $q\in\N$ for parity reasons,
we then get the expansion \cref{exppsiquant1Ut},
with first coefficient satisfying
$a_0(t,x)=\bar\mu^{-1}_{-t}(x)$,
for all $x\in X$ and $t\in\R$.
This implies formula \cref{exppsiquant1coeffUt} via
the formula \cref{mubar} for $\mu\in\cinf(X,\C)$.
\end{proof}

Recall the non-degeneracy assumption of \cref{nondegdef}.
We also have the following semi-classical trace formula for the
quantum evolution operator,
where we use the notations of \cref{sctrth}.

\begin{theorem}\label{sctrthcor}
Let $t\in\R$ be such that the fixed point set $X^{\varphi_t}$
of $\varphi_t:X\fl X$
is non-degenerate, and write $\{X_j\}_{1\leq j \leq m}$ 
for the set of its connected components. Then there exist
densities $\nu_{r}$ over $X^{\varphi_t}$
for any $r\in\N$
such that for any $k\in\N^*$ and as $p\fl +\infty$,
\begin{multline}\label{indeqfleUt}
\Tr\left[\exp\left(-2\pi\sqrt{-1}tpQ_p(f)\right)\right]\\
=\sum_{j=1}^q p^{\dim X_j/2}e^{-2\pi\i p\lambda_j}
\left(\sum_{r=0}^{k-1} p^{-r} \int_{X_j}\nu_{r}
+O(p^{-k})\right)\,,
\end{multline}
where $e^{2\pi\i \lambda_j}$ is the constant value of $\varphi^L$
over $X_j$, for some $\lambda_j\in\R$.
Furthermore, we have
\begin{equation}\label{nu0Ut}
\nu_{0}=
\left(\det(\overline{\Pi}_{t}^0)^{-1}
\tau_{t}^{K_X}\right)^{-\frac{1}{2}}
\det{}^{-\frac{1}{2}}_N
\left[P^N(\Pi_{0}^t-d\varphi^{-1}_t\overline{\Pi_t^0})
(\Id_{TX}-d\varphi_t)P^N\right]|dv|_{TX/N}\,,
\end{equation}
for some natural choices of square roots.
\end{theorem}
\begin{proof}
Using \cref{quantevolem}, this is a straightforward consequence
of \cref{sctrth}.
\end{proof}
The coefficient $\lambda_j\in\R$ appearing in the
expansion \cref{indeqfleUt} has a natural geometric interpretation,
which fits in a more general context.
In fact, note that the evolution equations \cref{Hamflow}
and \cref{liftflow} generalize to \emph{time-dependent Hamiltonians}
$F\in\cinf(\R\times X,\R)$, so that $f\in\cinf(X,\R)$ is replaced
by $f_t\in\cinf(X,\R)$ depending on $t\in\R$, with $f_t(x):=F(t,x)$ for
all $x\in X$. In that case, we also get
a Hamiltonian flow $\varphi_t:X\fl X$ together with a lift
$\varphi_t^L$ to the total space of $L$ preserving metric
and connection. 
The main difference here is that the term
replacing $\left(\nabla^L_{\xi_{f}}-2\pi\sqrt{-1} f\right)$
in the Kostant formula \cref{Kostantfla} will depend on time,
and the corresponding exponentiation as in \cref{alphaf}
at $x\in X$ reads
\begin{equation}\label{alphaft}
(\varphi_{t,p}^{-1}\tau_{t,p})_x=
e^{-2\pi\i p \int_0^t f_s(\varphi_s(x))ds}\,.
\end{equation}
Let now $x\in X$ and $t\in\R$ be such that $\varphi_t(x)=x$, and
via the canonical identification $L_x\otimes L_x^*\simeq\C$,
let us write
\begin{equation}
\varphi_{t,x}^L=:e^{2\i\pi\lambda_t(x)}\,.
\end{equation}
Note that the path $s\mapsto\varphi_s(x)$ defines a loop
$\gamma$ inside $X$. Assuming that this loop bounds an immersed
disk $D\subset X$, by the prequantizaton condition \cref{preq} and
via \cref{alphaft} above, we get
\begin{equation}
\lambda_t(x)=\int_D\om+\int_0^t f_s(\varphi_s(x))ds\,.
\end{equation}
This is a familiar quantity in symplectic topology, called the
\emph{action} of $F$ around the loop $\gamma$. 
Recall that
as $\varphi^L$ preserves the connection $\nabla^L$, the quantity
$\lambda_t(x)\in\R$ is constant when $x\in X$ varies continuously
over a submanifold of fixed points of $\varphi_t$.
This discussion motivates the following definition.

\begin{defi}\label{actiondef}
Let $F\in\cinf(\R\times X,\R)$ be a time-dependent Hamiltonian,
and let $t\in\R$ be such that its Hamiltonian flow $\varphi_t:X\fl X$
at time $t\in\R$ has non-degenerate fixed point set $X^\varphi\subset X$
in the sense of \cref{nondegdef}. Then for any connected component
$Y$ of $X^\varphi$, the associated real number $\lambda_0\in\R$
defined over $Y$ by
\begin{equation}
\varphi_t^L=:e^{2\i\pi\lambda_0}
\end{equation}
is called the \emph{action} of $F$ over $Y$.
\end{defi}
%We then see that the quantity $\lambda_j\in\R$ appearing in the
%expansion \cref{indeqfleUt}, for all $1\leq j\leq m$,
%is in fact the action of $\varphi_t$ over $X_j$.
In the general case of a time-dependent Hamiltonian
$F\in\cinf(\R\times X,\R)$, we get a
quantum Hamiltonian as before
from the corresponding quantized Kostant formula \cref{KostantflaHp},
and we can define the associated evolution operator by
equation \cref{quantevodef}. Then the analogue of
\cref{quantevolem} holds in that
case, and the analogues of \cref{Utth} and \cref{sctrthcor}
hold in the same way.

Finally, the situation described above also generalizes to the
case when $(L^p,h^{L^p},\nabla^{L^p})$ over $\R\times X$
is replaced by $E_p=L^p\otimes E$ with induced metric and connection,
where $(E,h^E,\nabla^E)$ is a
Hermitian vector bundle with connection
over $\R\times X$. In that case, we make the further assumption
that the Hamiltonian
flow $\varphi_t:X\fl X$ lifts to a bundle map
$\varphi_t^E:E_0\rightarrow E_t$ over $X$
preserving metric and connection for all $t\in\R$, and we still write
$\varphi_{t,p}$ for the corresponding action on $E_p$ for any $p\in\N$.
%,
%and for all $x\in X$, set
%\begin{equation}\label{tautpE}
%\tau_{p,t}:E_{p,0,x}\longrightarrow E_{p,t,\varphi_t(x)}\,,
%\end{equation}
%for the parallel transport along the path $s\mapsto(s,\varphi_s(x))$
%for $s\in[0,t]$.
%, we again obtain the
%quantum evolution operator
%$\varphi^{*}_{t,p}\Tau_{p,t}\in\End(\HH_{p,0})$,
Then the analogue of \cref{quantevolem} for the quantum evolution
operator defined by equation \cref{quantevodef} holds as before,
and using the general
setting of \cref{setting}, we also get the corresponding analogues of
\cref{Utth} and \cref{sctrthcor}.
The case of $E=K_X^{1/2}$,
with the lift induced by $\varphi_t^{K_X}:K_{X,0}\fl K_{X,t}$ for all
$t\in\R$, will be of particular interest in the next section.
\section{Gutzwiller trace formula}\label{gutzsec}

Consider the setting of the previous section, with
Hamiltonian $f\in\cinf(X,\R)$ not depending on time,
and a Hermitian vector bundle with connection $(E,h^E,\nabla^E)$
over $\R\times X$, so that $L^p$ is replaced by $E_p=L^p\otimes E$
for all $p\in\N$.
Let $g:\R\fl\R$ be a smooth function with compact support, and
for all $t\in\R$, set
\begin{equation}\label{Fourierg}
\hat{g}(t):=\int_\R g(t)e^{-2\pi\i t}\,dt\,.
\end{equation}
We define
the family of operators $\{\hat{g}(pQ_p(f))\in\End(\HH_p)\}_{p\in\N}$ by
the formula
\begin{equation}\label{psiquant}
\begin{split}
\hat{g}(pQ_p(f)):=\int_\R g(t)
\left(\varphi_{t,p}^*\Tau_{p,t}\right) dt\,.
\end{split}
\end{equation}
In the particular case of $E=\C$, so that \cref{quantevolem} holds,
we get
\begin{equation}\label{psiquantformal}
\hat{g}(pQ_p(f))=\int_\R g(t)\exp(-2\pi\sqrt{-1}tpQ_p(f)) dt\,,
\end{equation}
recovering the usual definition via functional calculus
from \cref{Fourierg}.
Note that formula \cref{psiquantintro} for $\hat{g}(pQ_p(f-c))$
with $c\in\R$
follows from \cref{Hamflow,liftflow,pullbackdef}, as
replacing $f\in\cinf(X,\R)$ by $f-c$
does not change the Hamiltonian flow $\varphi_t:X\fl X$
but multiplies its lift to $L$ by $e^{-2\pi\sqrt{-1}tc}$.
In the context of semi-classical analysis, the
\emph{Gutzwiller trace formula} predicts a semi-classical estimate
for the trace
$\Tr[\hat{g}(pQ_p(f-c))]$ as $p\fl+\infty$, where $c\in\R$ is a regular
value of $f$, showing that it localizes
around the periodic orbits of the Hamiltonian flow of $f$
inside the level set $f^{-1}(c)$.

The following preliminary result shows that the Schwartz kernel of
$\hat{g}(pQ_p(f-c))$ decreases rapidly outside
$f^{-1}(c)$ as $p\fl+\infty$.

\begin{prop}\label{oscvener}
Let $c\in\R$ be a regular value of $f\in\cinf(X,\R)$.
For any $k\in\N$, there exists $C_k>0$ such that
for any $y\in X,\,x\in X\backslash f^{-1}(c)$ and $p\in\N$, we have
\begin{equation}\label{estoscvener}
\left|\hat{g}(pQ_p(f-c))(x,y)\right|_p \leq \frac{C_k}{|f(x)-c\,|^k}
p^{n-\frac{k}{2}}\,.
\end{equation}
\end{prop}
\begin{proof}
% By \cref{alphaf}, we have $\lambda(t)=e^{-2\pi\sqrt{-1} Ct}$
%for all $t\in\R$.
To simplify notations, we assume $E=\C$, the case of general $E$
being completely analogous.
Recall the definition \cref{tautp} of $\tau_{t,p}$,
and $p\in\N$. For any $x\in X$ and $t_0\in\R$,
consider a chart around $x_0:=\varphi_{t_0}(x)\in X$
as in \cref{chart}, such that the radial line generated by
$\xi_{f,x_0}$ in $B^{T_{x_0}X}(0,\epsilon)$ is sent to the
path $s\mapsto\varphi_s(x_0)$ in $V_{x_0}$.
Then $L^p$ is identified with $L^p_{x_0}$ along this path
by the parallel transport $\tau_{t,p}$, for all $p\in\N$ and
$|t|<\epsilon$. Thus for any $Z\in B^{T_{x_0}X}(0,\epsilon)$ sent to
$y\in V_{x_0}$ in the chart \cref{chart}
and for any $t\in\R$ small enough, we have
\begin{equation}
\tau_{t,p}^{-1}\Tau_{p,t_0+t}(\varphi_{t_0+t}(x),y)=
\Tau_{p,t_0+t,x_0}(t\xi_{f,x_0},Z)\,.
\end{equation}
Using $\tau_{t_0+t,p}=\tau_{t,p}\tau_{t_0,p}$
for all $t_0\in T$ and $|t|<\epsilon$, we can apply
\cref{BTasy} in such charts for all $t_0\in\R$,
so that for any $k\in\N$ we get $C_k>0$ such that for any
$x,\,y\in X,\,t\in\Supp g$ and all $p\in\N$, we have
\begin{equation}\label{majdtTau}
\left|\Dk{k}{t}\tau_{t,p}^{-1}\Tau_{p,t}(\varphi_t(x),y)
\right|_p \leq C_k p^{n+\frac{k}{2}}\,.
\end{equation}
Recall from \cref{hamiltonian} that the Hamiltonian flow of $f$ is
the same as the Hamiltonian flow of $f-c$, for all $c\in\R$.
Then exponentiating Kostant formula as in \cref{alphaf}
and as $\Supp g$ is compact,
we can integrate by parts
to get from the definition \cref{psiquant} of $\hat{g}(pQ_p(f))$ that
for any $x,y\in X$ and $k\in\N$,
\begin{multline}\label{oscvenerfla}
\hat{g}(pQ_p(f-c))(x,y)=\int_\R g(t)\tau_{t,p}^{-1}
\Tau_{p,t}(\varphi_t(x),y) e^{-2\pi\sqrt{-1} tp(f(x)-c)} dt\\
=\frac{1}{(-2\pi\sqrt{-1} p(f(x)-c))^k}\int_\R g(t)
\tau_{t,p}^{-1}\Tau_{p,t}(\varphi_t(x),y) \Dk{k}{t}
e^{-2\pi\sqrt{-1} tp(f(x)-c)} dt\\
=(2\pi\sqrt{-1})^{-k}
\frac{p^{-k}}{(f(x)-c)^k}\int_\R \Dk{k}{t}
\left(g(t)\tau_{t,p}^{-1}\Tau_{p,t}(\varphi_t(x),y)
\right)e^{-2\pi\sqrt{-1} tp(f(x)-c)} dt\,.
\end{multline}
This proves the result by \cref{majdtTau}.
\end{proof}

Let us now estimate the Schwartz kernel of $\hat{g}(pQ_p(f-c))$
as $p\fl+\infty$. \cref{oscvener} shows that it localizes
around the level set $f^{-1}(c)$, in contrast with \cref{Utth}.
In the following theorems and their proofs,
we will use freely the notations of \cref{intro,setting}.

\begin{theorem}\label{exppsiquant}
Let $c\in\R$ be a regular value of $f\in\cinf(X,\R)$.
If $x,\,y\in X$ do not satisfy $\varphi_t(x)=y$ for some $t\in\R$
or do not satisfy $f(x)=f(y)=c$,
then for any $k\in\N$, there exists $C_k>0$ such that for all
$p\in\N$,
\begin{equation}\label{exppsiquant3}
\left|\hat{g}(pQ_p(f-c))(x,y)\right|_p<C_k p^{-k}\,.
\end{equation}
Let $x,\,y\in f^{-1}(c)$, and write $T:=\{t\in\R~|~\varphi_t(x)=y\}$.
Then there exist $b_{t_0,r}\in\C$ for all $r\in\N$ and $t_0\in T$
such that for any $k\in\N^*$ and as $p\fl+\infty$,
\begin{equation}\label{exppsiquant1}
\hat{g}(pQ_p(f-c))(x,y)=p^{n-\frac{1}{2}}\sum_{t_0\in T}g(t_0)
\left(\sum_{r=0}^{k-1} p^{-r}b_{t_0,r} + O(p^{-k})\right)
\left(\varphi^{E,-1}_{t_0}\tau^E_{t_0}\right)\tau_{t_0,p}^{-1}\,.
\end{equation}
Furthermore, for any $t_0\in T$, we have
\begin{equation}\label{exppsiquant1coeff}
b_{t_0,0}^2=
\left(\det(\overline{\Pi}_{t_0}^0)^{-1}\tau^{K_X}_{t_0}\right)^{-1}_x
\<\Pi_0^{t_0}\xi_{f,x},\xi_{f,x}\>_{g^{TX}_0}^{-1}\,.
\end{equation}
In particular, it satisfies $b_{t_0,0}\neq 0$, for all $t_0\in T$.
\end{theorem}
\begin{proof}
Note first that the estimate \cref{exppsiquant3} is a
straightforward consequence of either \cref{BTasy} or \cref{oscvener}
respectively.

To establish the expansion \cref{exppsiquant1},
fix $x,\,y\in X$ such that $f(x)=f(y)=c\in\R$ is a regular value of
$f$. Then $f(x)-c=0$, and exponentiating the Kostant formula
as in \cref{alphaf}, we can
write the Schwartz kernel of $\hat{g}(pQ_p(f-c))$ as
\begin{equation}\label{kerpsi}
\hat{g}(pQ_p(f-c))(x,y)=\int_\R g(t)\tau_{t,p}^{-1}\varphi^{E,-1}_t
\Tau_{p,t}(\varphi_t(x),y) dt\,.
\end{equation}
As $c\in\R$ is a regular value of $f$,
the Hamiltonian vector field $\xi_f$ does not vanish over
$f^{-1}(c)$.
By \cref{thetafla}, this implies that
for any $\theta\in\,]0,1[$ and $\epsilon>0$,
we get the following estimate as $p\fl+\infty$,
\begin{equation}\label{locpath}
\hat{g}(pQ_p(f-c))(x,y)=\sum_{t_0\in T} \int_{t_0-\epsilon p^{-\frac{\theta}{2}}}^{t_0+\epsilon p^{-\frac{\theta}{2}}} g(t)
\tau_{t,p}^{-1}\varphi^{E,-1}_t
\Tau_{p,t}(\varphi_t(x),y) dt+O(p^{-\infty})\,,
\end{equation}
where all terms but a finite number
vanish by compacity of $\Supp g$.

Consider a chart around $y$ as in \cref{chart},
sending the radial line generated by $\xi_{f,y}$ in
$B^{T_yX}(0,\epsilon)$
to the path $s\mapsto\varphi_s(y)$ in $V_y$.
%Let
%$\rho\in\cinf(]-\epsilon,\epsilon[,\R)$ be the function
%defined for all $t\in\,]-\epsilon,\epsilon[$ by the relation
%\begin{equation}
%\rho(t)t\xi_f=\varphi_t(x)
%\end{equation}
%in these coordinates, so that in particular $\rho(0)=|\xi_{f,x}|$.
Then $L^p$ is identified with $L^p_y$ along this path
by the parallel transport $\tau_{t,p}$, for all $p\in\N$ and
$|t|<\epsilon$. Using $\tau_{t_0+t,p}=\tau_{t,p}\tau_{t_0,p}$
for all $t_0\in T$ and $|t|<\epsilon$, we can then
apply \cref{BTasy} to get a family
$\{G_{r,t,y}(Z,Z')\in E_{t,y}\otimes E_{0,y}^*\}_{r\in\N}$
of polynomials in $Z,Z'\in T_yX$
of the same parity as $r$ and smooth in $t\in\R$,
such that for any $\delta\in\,]0,1[$ and $k\in\N^*$, there is
$\theta\in\,]0,1[$ such that for all $p\in\N$,
\begin{multline}\label{kergq}
\hat{g}(pQ_p(f-c))(x,y) =\sum_{t_0\in T} p^n \tau_{p,t_0}^{-1}
\int_{-\epsilon p^{-\frac{\theta}{2}}}^{\epsilon p^{-\frac{\theta}{2}}}
g(t_0+t)\varphi^{E,-1}_{t_0+t}\\
\sum_{r=0}^{k-1} p^{-\frac{r}{2}}
G_r\cT_{t_0+t,y}(\sqrt{p}t\xi_{f,y},0) dt
+ p^{n-\frac{\theta}{2}}O(p^{-\frac{k}{2}+\delta})\,.
\end{multline}
%Let $d_r\in\N$ be the supremum degree of $G_{r,t}(Z,0)$ for all
%$t\in\Supp(g)$, which exists by \cref{BTasy}, and write
%$a_{r,k}(t)$ for its $k$-th coefficient.
%For any $t_0\in T\cap\Supp(g)$,
%consider the following Taylor expansion in $t\in\R$ up to order
%$l\in\N$,
%\begin{equation}\label{Taylorcoeff}
%\begin{split}
%p^{-\frac{r}{2}} & G_{r,t_0+t}(\sqrt{p}\rho(t)te_1,0)
%=p^{-\frac{r}{2}}\sum_{k=0}^{d_r} a_{r,k}(t_0+t)\rho(t)^k(\sqrt{p}t)^k\\
%& =p^{-\frac{r}{2}}\sum_{k=0}^{d_r} \left(\sum_{j=0}^l\Dkk{j}
%{a_{r,k,t_0}h^k_0}{t} t^j+O(|t|^{l+1})\right)(\sqrt{p}t)^k\\
%& =\sum_{k=0}^{d_r}\sum_{j=0}^l p^{-(r+j)/2}
%\Dkk{j}{a_{r,k,t_0}h^k_0}{t} (\sqrt{p}t)^{k+j}
%+p^{-(l+1+k)/2}O(|\sqrt{p}t|^{l+k+l})\,.
%\end{split}
%\end{equation}
%By assumption, the only non vanishing terms of the sum \cref{Taylorcoeff} are those for which $k\in\N^*$ is of the same parity
%as $r$, so that the exponents $k+j$ appearing in the last line of
%\cref{Taylorcoeff} are of the same parity as $r+j$ for any
%$k,\,j\in\N$.
%
%On the other hand, by \cref{locmodT} and the same reasoning as \cref{TaylorPP}-\cref{PPphi=PPdphi}, we get a family $\{a_{k,j}\}_{k,j\in\N}$ of complex numbers, such that for any $l\in\N$ and $j\in\Z$,
%\begin{multline}\label{locmodgutz}
%\PP_{t+t_j}\PP_0(0,\sqrt{p}h(t)t)=\det(A_0^{t+t_j})^{\frac{1}{2}}\exp\left(-p\pi\<\Pi_0^{t+t_j} \xi_{f,x},\xi_{f,x}\>h(t)^2t^2\right)\\
%\sum_{k=0}^l a_{k,j} p^{-\frac{k}{2}} (\sqrt{p}t)^k+p^{-(l+1)/2}O(|\sqrt{p}t|^{l+1}).
%\end{multline}
Consider the right hand side of \cref{kergq},
and let us
apply the Taylor expansion in $t\in\R$
described in \cref{Taylordef}
on all terms depending on $t\in\R$
respectively inside and outside the exponential of
the local model \cref{locmodT} for $\cT_{t_0+t,y}$.
We then get $F_{r,t_0}(t)\in E_{t_0}\otimes E_0^*$, polynomial in
$t\in\R$ and of the same parity as $r$ for any $r\in\N$,
such that for all $t_0\in T$ and as $|t|\fl 0$,
\begin{multline}\label{Taylorexp}
\sum_{r=0}^{k-1} p^{-\frac{r}{2}}g(t_0+t)\varphi^{E,-1}_{t_0+t}
G_r\cT_{t_0+t,y}(\sqrt{p}t\xi_{f,y},0)=\\
g(t_0)\varphi^{E,-1}_{t_0}
\sum_{r=0}^{k-1}p^{-r/2}F_{r,t_0}(\sqrt{p}t)
\exp\left(-p\pi\<\Pi_0^{t_0}\xi_{f,y},\xi_{f,y}\>t^2\right)
+p^{-\frac{k}{2}}O(|\sqrt{p}t|^{M_k})\,,
\end{multline}
for some $M_k\in\N^*$ depending on the degrees of the polynomials
$G_{r,t_0+t,y}$ for all $|t|<\epsilon$, $t_0\in T\cap\Supp g$
and $1\leq r\leq k$. Furthermore, from the formula
\cref{|J|0} for the first coefficient $G_{0,t,y}$, we know that
$F_{0,t_0}(t)=\bar\mu_{t_0}^{-1}(y)\tau_{t_0,y}^E$
for all $t\in\R$.
Note that by the definition \cref{Pi0t} of $\Pi_0^{t_0}$,
the real part of
$\<\Pi_0^{t_0}\xi_{f},\xi_{f}\>$ is strictly positive,
so that the right hand side of
\cref{Taylorexp} decreases  
exponentially in $t\in\R$. 
Writing
\begin{equation}\label{delta'}
\delta_k=\delta+\frac{M_k(1-\theta)}{2}\,,
\end{equation}
and after a change of variable $t\mapsto t/\sqrt{p}$, we then get
\begin{multline}\label{kergq1}
\hat{g}(pQ_p(f-c))(x,y)=
\sum_{t_0\in T}g(t_0)\bar{\mu}_{t_0}^{-1}(y)
\varphi^{E,-1}_{t_0}\tau^E_{t_0}
\tau_{p,t_0}^{-1}
p^{n-\frac{1}{2}}\\
\sum_{r=0}^{k-1} p^{-\frac{r}{2}}
\int_\R F_{r,t_0}(t)\exp\left(-\pi\<\Pi_0^{t_0}\xi_{f,y},
\xi_{f,y}\>t^2\right) dt+p^{n-\frac{1}{2}}O(p^{-\frac{k}{2}+\delta_k})\,.
\end{multline}
As $F_{2q+1,t_0}(t)$ is odd as a function of $t\in\R$ for all
$q\in\N$, we get
\begin{equation}\label{odd}
\int_\R F_{2q+1,t_0}(t)\exp\left(-\pi\<\Pi_0^{t_0}\xi_{f,y},
\xi_{f,y}\>t^2\right) dt=0\,.
\end{equation}
As $\delta_k\fl \delta$ when $\theta\fl 1$ by \cref{delta'}
and as $\delta$ can be chosen arbitrary small,
this gives the expansion \cref{exppsiquant1},
and we get the formula \cref{exppsiquant1coeff} for the first
coefficient via the classical formula for Gaussian integrals,
using the formula \cref{mubar} and the fact that
$F_{0,t_0}(t)=\bar\mu_{t_0}^{-1}(y)\tau_{t_0}^E$ for all $t\in\R$.
\end{proof}

%Fix now $c\in\R$ to be a regular value of $f$, and let 
%$T\subset\R$ be defined by
%\begin{equation}
%T:=\{t\in\R~|~\exists~x\in f^{-1}(c)~
%\text{such that}~\varphi_t(x)=x\}\,.
%\end{equation}
%Assume that $T\cap\Supp(g)$ is finite, and that the fixed point
%set of $\varphi_t$ is
%non-degenerate in a neighborhood of $f^{-1}(c)$
%for any $t\in T\cap\Supp(g)$, with
%fixed point set $X^{\varphi}_t$ transverse to $f^{-1}(c)$.

Consider now the hypotheses and notations of
\cref{gutzwillerperorbintro}. Recall that
$c\in\R$ is a regular
value of $f$, so that $\Sigma:=f^{-1}(c)$ is a smooth manifold.
Then 
there exists $\epsilon>0$
and diffeomorphisms
\begin{equation}\label{idgal}
\Psi_{t_0}:
f^{-1}(]c-\epsilon,c+\epsilon[)\xrightarrow{~~\sim~~}\,
]c-\epsilon,c+\epsilon[\,\times\,\Sigma\,
\end{equation}
for all $t_0\in T$,
such that $f(u,x)=u$ for all
$(u,x)\in\,]c-\epsilon,c+\epsilon[\,\times\,\Sigma$
under this identification, and such that
\begin{equation}\label{idgalYt}
\Psi_{t_0}\left(X^{\varphi_{t_0}}
\cap f^{-1}(]c-\epsilon,c+\epsilon[)\right)
=\bigcup_{\substack{1\leq j\leq m\\ t_j=t_0}}~
]c-\epsilon,c+\epsilon[\,\times\, Y_j\,,
\end{equation}
where the connected components $Y_j$ of
$X^{\varphi_{t_j}}\cap f^{-1}(c)$
are seen as submanifolds of $\Sigma$, for all $1\leq j\leq m$.
%To simplify the notations, we extend $T$ to be the finite set of
%connected components of $Y_t$ for all $t\in\Supp(g)$,
%so that the times $t\in T$ are counted "with multiplicity", and
%$Y_t$ will be connected submanifolds of $\Sigma$ for all $t\in T$.
%Note that if $g$ has support in a small enough neighborhood
%of $0\in\R$, then $T\cap\Supp(g)$ is reduced to $0$ and
%the assumptions above are automatically satisfied.
We endow $\Sigma$ with the Riemannian metric $g^{T\Sigma}$ induced
by $g^{TX}_0:=\om(J_0\cdot,\cdot)$ via the inclusion
$\Sigma=f^{-1}(c)\subset X$.
For any $1\leq j\leq m$, let $|dv|_{Y_j}$
be the Riemannian density over $Y_j$ induced by $g^{T\Sigma}$
and let $N$ be the normal bundle
of $Y_j$ inside $\Sigma$. We write $P^N:T\Sigma\fl N$ for the
orthogonal projection with respect to $g^{T\Sigma}$ over $Y_j$
for all $1\leq j\leq m$.

Recall that the \emph{Liouville measure}
on the level set $f^{-1}(c)$ is induced by the volume form
\begin{equation}\label{liouvilleham}
\frac{\iota_{v}\,\om^n}{(n-1)!}\in\Om^{2n-1}(f^{-1}(c),\R)\,,
\end{equation}
for any $v\in\cinf(f^{-1}(c),TX)$ satisfying $\om(\xi_f,v)=1$,
and does not depend on such a choice. We write $\Vol_\om(f^{-1}(c))>0$
for the volume of $f^{-1}(c)$ with respect to \cref{liouvilleham}.
Recalling \cref{actiondef},
the following theorem is a version of the
\emph{Gutzwiller trace formula} in geometric quantization of
compact prequantized symplectic manifolds, and
is the main result of this section.

\begin{theorem}\label{gutzwiller}
Under the above assumptions, there exist $b_{j,r}\in\C$
for all $r\in\N$ and $1\leq j\leq m$,
%, depending only on $g$ around $0$ and on all the local data around $Y$, 
such that for any $k\in\N^*$, we have as $p\fl+\infty$,
\begin{equation}\label{gutzexp}
\Tr\left[\hat{g}(pQ_p(f-c))\right]=\sum_{j=1}^m
p^{(\dim Y_j-1)/2}g(t_j)
e^{-2\pi\i p \lambda_j}\left(\sum_{r=0}^{k-1}
p^{-r}b_{j,r}+O(p^{-k})\right)\,,
\end{equation}
where $\lambda_j\in\C$ is the action of $f$
over $Y_j$. Furthermore, for all $1\leq j\leq m$ we have
\begin{multline}\label{b0gutz}
b_{j,0}=\int_{Y_j}\Tr_{E}[\varphi^{E,-1}_{t_j}\tau^{E}_{t_j}]\,
\left(\det(\overline{\Pi}_{t_j}^0)^{-1}
\tau_{t_j}^{K_X}\right)^{-\frac{1}{2}}\\
\det{}^{-\frac{1}{2}}_{N}
\left[P^N(\Pi_{0}^{t_j}-d\varphi^{-1}_{t_j}\overline{\Pi}_{t_j}^0)
(\Id_{TX}-d\varphi_{t_j})P^N\right]\frac{|dv|_{Y_j}}{|\xi_f|_{g^{TX}_0}}
\,,
\end{multline}
for some natural choices of square roots.
In particular, we have as $p\fl+\infty$,
\begin{equation}\label{Weylterm}
\Tr\left[\hat{g}(pQ_p(f-c))\right]=p^{n-1}g(0)\rk(E)
\Vol_\om(f^{-1}(c))
+O(p^{n-2})\,.
\end{equation}
\end{theorem}
\begin{proof}
First note that replacing $f\in\cinf(X,\R)$ by $f-c$,
we are reduced to the case $c=0$.
Consider thus $f\in\cinf(X,\R)$ satisfying the hypotheses of
\cref{gutzwillerperorbintro} with $c=0$.
Using the trace
formula \cref{Trfla} and the definition of $\hat{g}(pQ_p(f))$ in \cref{psiquant} for all $p\in\N$, we know that
\begin{equation}\label{trpsiquant}
\begin{split}
\Tr\big[\hat{g}(pQ_p(f))\big]
&=\int_X \Tr\left[\hat{g}(pQ_p(f))(x,x)\right]dv_X(x)\\
&=\int_X\int_\R g(t)\Tr\left[
\varphi_{t,p}^{-1}\Tau_{p,t}(\varphi_t(x),x)\right]dt\,dv_X(x)\,.
\end{split}
\end{equation}
For any $\epsilon>0$, write
\begin{equation}
U(\epsilon):=f^{-1}(]-\epsilon,\epsilon\,[)\,.
\end{equation}
Then by \cref{oscvener}, for any $\theta\in\,]0,1[$ and $k\in\N$,
we get $C_k>0$ such that for all
$x\in X\backslash U(\epsilon p^{-\frac{\theta}{2}})$,
\begin{equation}\label{estUepsphastat}
\left|\hat{g}(pQ_p(f))(x,x)\right|_p\leq \frac{C_k}{\epsilon^k}
p^{n-\frac{k(1-\theta)}{2}}\,,
\end{equation}
so that in particular, we have as $p\fl+\infty$,
\begin{equation}\label{trpsiloc1}
\Tr\big[\hat{g}(pQ_p(f))\big]=
\int_{U(\epsilon p^{-\theta/2})}\Tr\left[
\hat{g}(pQ_p(f))(x,x)\right]dv_X(x)+O(p^{-\infty})\,.
\end{equation}
Recall that
$T:=\{t\in\Supp g~|~\exists\,x\in f^{-1}(0),~\varphi_t(x)=x\}$
is finite, and let $\epsilon>0$ be such that all
$u\in\,]-\epsilon,\epsilon[$ are regular values of $f$, so that
the Hamiltonian vector field $\xi_f$ does not vanish on 
the closure of $U(\epsilon)$. Then in the same way as in
\cref{locpath}, we get from the rapid decrease \cref{thetafla}
of $\Tau_{p,t}(\cdot,\cdot)$ outside the diagonal
that as $p\fl+\infty$,
\begin{multline}\label{gpfla1}
\Tr\big[\hat{g}(pQ_p(f))\big]=\\
\sum_{t_0\in T}
\int_{U(\epsilon p^{-\theta/2})}
\int_{t_0-\epsilon p^{-\theta/2}}^{t_0+\epsilon p^{-\theta/2}}
g(t)\Tr\left[\varphi_{t,p}^{-1}
\Tau_{p,t}(\varphi_t(x),x)\right]dt\,dv_X(x)+O(p^{-\infty})\,.
\end{multline}
Take $\epsilon>0$ small enough so that the identification
$\Psi_{t_0}$ of \cref{idgal} holds for any $t_0\in T$.
%. As $c\in\R$ is a regular value of $f$,
%we can choose $\epsilon>0$ so small
%and a diffeomorphism
%\begin{equation}\label{idgal}
%f^{-1}(]c-\epsilon,c+\epsilon[)\longrightarrow\,
%]c-\epsilon,c+\epsilon[\,\times\Sigma\,,
%\end{equation}
%sending $f^{-1}(u)$ on $\{u\}\times\Sigma$ for all
%$u\in\,]c-\epsilon,c+\epsilon[$ and such that
Recall from \cref{hamiltonian} that the Hamiltonian vector field
$\xi_f$ is tangent to the level sets of $f$. Then for any
$t_0\in T$, we get
diffeomorphisms $\varphi_{u,t}:\Sigma\fl\Sigma$, depending
smoothly on
$u\in\,]-\epsilon,\epsilon[$ and $t\in\,]t_0-\epsilon,t_0+\epsilon[$,
such that
\begin{equation}\label{varphigal}
\varphi_t(u,x)=(u,\varphi_{u,t}(x))~~\text{and}~~
\varphi_{u,t_0}(x)=x~~\text{for all}~~x\in Y_j\,,
\end{equation}
in the coordinates $(u,x)\in\Psi_{t_0}(U(\epsilon))$ of \cref{idgal}
and for all $1\leq j\leq m$ such that $t_j=t_0$.
%Furthermore, by \cref{nondegdef}, we can choose \cref{idgal} so that
%\begin{equation}
%Y_{t_0}\times\,]-\epsilon,\epsilon[~\xrightarrow{~~\sim~~}
%\ \{x\in U(\epsilon)~|~\varphi_{t_0}(x)=x\}\,.
%\end{equation}
%via this identification.
%Then via this identification, we can cover the fixed point set
%$Y_{t_0}\subset X$ fo $\varphi_{t_0}$ by open sets of the form
%$V\times\,]-\epsilon,\epsilon[\subset X$ as above,
%and the integral \cref{gpfla1} localizes around this
%neighborhood up to $O(p^{-\infty})$ by \cref{thetacongUt}.
%We are the reduced to evaluate the integral \cref{gpfla1}
%over such open sets $V\times\,]-\epsilon,\epsilon[$.
For any $\epsilon>0$ and $1\leq j \leq m$,
consider the normal geodesic neighbourhood
$V_{j}(\epsilon)\subset\Sigma$ of
$Y_{j}$ inside $(\Sigma,g^{T\Sigma})$.
Then by the non-degeneracy assumption of \cref{nondegdef}
and as all $u\in\,]-\epsilon,\epsilon[$ are regular values of $f$,
the map $\Phi:(u,t,x)\mapsto(u,t,\varphi_{u,t}(x))$ is also
non-degenerate
around the fixed point set
$]-\epsilon,\epsilon[\,\times\,\{t_j\}\times Y_j$ inside
$]-\epsilon,\epsilon[\,\times\,\R\times\Sigma$,
so that working in local charts, we see that
there exists $\epsilon'>0$ such that
for all $\theta\in\,]0,1[$ and $p\in\N$, 
\begin{multline}\label{coord}
d^{\Sigma}(x,\varphi_{u,t}(x))>\epsilon' p^{-\theta/2}\quad~~
\text{as soon as}\\
(t,x)\,\notin\bigcup_{1\leq j\leq m}~
]t_j-\epsilon p^{-\theta/2},t_j+\epsilon p^{-\theta/2}[\,
\,\,\times\,\,
V_{j}(\epsilon p^{-\theta/2})\,,
\end{multline}
for all $u\in\,]-\epsilon,\epsilon[$.
%so that in particular, we see from \cref{Approx}
%that we are reduced to consider the integral of \cref{gpfla1} for
%\begin{equation}\label{coord}
%(x,u,t)\in 
%V_{t_0}(\epsilon p^{-\theta/2})\,\times
%]c-\epsilon p^{-\theta/2},c+\epsilon p^{-\theta/2}\,[
%\times\,]t_0-
%\epsilon p^{-\theta/2},t_0+\epsilon p^{-\theta/2}[\,
%\end{equation}
%via the identification \cref{idgal}.
On the other hand, as $f(x,u)=u$
in the coordinates $(u,x)\in\Psi_{t_j}(U(\epsilon))$ of \cref{idgal},
by the definition \cref{hamiltonian} of the Hamiltonian vector field
$\xi_f$ and the definition \cref{gTXintro}
of $g^{TX}_0$,
we get a function $\varrho\in\cinf(U(\epsilon),\R)$
such that over $U(\epsilon)$, we have
\begin{equation}
dv_X=\varrho\,du\,dv_\Sigma\quad\text{and}\quad
\varrho(0,x)=|\xi_{f,x}|^{-1}_{g^{TX}_0}\,
~~\text{for all}~~x\in\Sigma\,.
\end{equation}
We can then rewrite \cref{gpfla1} as
\begin{multline}\label{Gpfla0}
\Tr\big[\hat{g}(pQ_p(f))\big]=\\
\sum_{j=1}^m
\int_{V_{j}(\epsilon p^{-\theta/2})}
\int_{-\epsilon p^{-\frac{\theta}{2}}}
^{\epsilon p^{-\frac{\theta}{2}}}
\int_{-\epsilon p^{-\frac{\theta}{2}}}
^{\epsilon p^{-\frac{\theta}{2}}}
\Tr\left[I_{j,p}(t,u,x)\right]dt\,du\,dv_\Sigma(x)+O(p^{-\infty})\,,
\end{multline}
where in the coordinates $(u,x)\in\Psi_{t_j}(U(\epsilon))$ and
for any $t\in\,]t_j-\epsilon,t_j+\epsilon[$, we set
\begin{equation}\label{Gfla1}
I_{j,p}(t,u,x)=
g(t_j+t)
\varphi_{t_j+t,p}^{-1}\Tau_{p,t_j+t}((u,\varphi_{u,t_j+t}(x)),(u,x))
\varrho(u,x)\,.
\end{equation}
Recall that $N$ denotes the normal bundle of
$Y_{j}$ inside $\Sigma$ equipped with
the metric $g^N$ induced by $g^{T\Sigma}$,
and consider the natural identification
\begin{equation}\label{idBN}
V_{j}(\epsilon)\xrightarrow{~~\sim~~}B^N(\epsilon):=\{w\in N~|~
|w|_N<\epsilon\}\,.
\end{equation}
As $\varphi_{t_j}(x)=x$ implies $\varphi_{t_j}(\varphi_t(x))=\varphi_t(x)$ for all $t\in\R$ and $x\in X$ by the $1$-parameter
group property of $\varphi_t$, we know that its flow is transverse
to the fibres of
the ball bundle $B^N(\epsilon)$ via the identification \cref{idBN}
for $\epsilon>0$ small enough. 
We can then pick $x_0\in Y_{j}$ and $u\in\,]-\epsilon,\epsilon[$,
and consider the natural embedding
defined from the fibre $B^{N}_{x_0}(\epsilon)$ of the ball bundle
\cref{idBN} over $x_0$
and a neighbourhood $I\subset\R$ of $0$ by
\begin{equation}\label{coordU}
\begin{split}
I\times B^N_{x_0}(\epsilon)&\xrightarrow{~~\sim~~}W_u\subset\Sigma\\
(t,w)&\longmapsto\varphi_{u,t_j+t}(w)\,.
\end{split}
\end{equation}
We identify in turn $W_u$ with a subset of $T_{x_0}\Sigma$ via
the inclusion
\begin{equation}\label{coordTSig}
\begin{split}
I\times B^N_{x_0}(\epsilon)&\longrightarrow T_{x_0}\Sigma\\
(t,w)&\longmapsto w+t\xi_{f,u}\,,
~~\text{where}~~\xi_{f,u}:=\D{t}\varphi_{u,t}(x_0)\in T_{x_0}\Sigma\,.
\end{split}
\end{equation}
For any $u\in\,]-\epsilon,\epsilon[$,
we identify $L$ over $W_u$ with the pullback of $L$ over
$B^N_{x_0}(\epsilon)$ via parallel transport with respect
to $\nabla^L$ along flow lines of $t\mapsto\varphi_{u,t}$,
then trivialize $L$ over $B^N_{x_0}(\epsilon)$
via parallel transport with respect to $\nabla^L$ along radial
lines. Then $L^p$ is trivialized by the parallel transport $\tau_{t,p}$
along the flow lines of $\varphi_t$, and by
the exponentiation of Kostant formula \cref{alphaf}, we have
in this trivialization,
\begin{equation}\label{tauL2}
\varphi_{t,p}^{-1}=e^{2\i\pi ptu}~~\text{for all}~~|t|<\epsilon
~~\text{and}~~p\in\N\,.
\end{equation}
Now by the definition \cref{hamiltonian} of $\xi_f$
and as $\{u\}\times\Sigma$ corresponds to the level set $f^{-1}(u)$
via the identification $\Psi_{t_j}$ of \cref{idgal}, we know
that for any vector field $v\in\cinf(W_u,T\Sigma)$, we have
\begin{equation}
R^L\left(v,\xi_f\right)=\frac{2\pi}{\i}\om\left(v,\xi_f\right)=0\,.
\end{equation}
This shows that the trivialization of $L$ described above
coincides with a trivialization along radial lines
of $W_u$ in \cref{coordTSig}, for all $u\in\,]-\epsilon,\epsilon[$,
so that we are under the hypotheses of \cref{BTasy}.
As in \cref{actiondef} and using \cref{idgalYt}
under the identification $\Psi_{t_j}$ of \cref{idgal}, define the action $\lambda_j\in\R$ by
\begin{equation}
\varphi_{t_j,p}
=:e^{2\pi\i p\lambda_j}~~\text{over}~~
]-\epsilon,\epsilon[\,\times\,Y_{j}\,.
\end{equation}
By a standard computation,
which can be found for example in \cite[(1.2.31)]{MM07}
and which holds in any trivialization of $L$ along radial lines,
the connection $\nabla^L$ at $w\in B^{N}_{x_0}(\epsilon)$
inside $W_u\subset T_{x_0}\Sigma$ as in \cref{coordTSig} has the form
\begin{equation}\label{nablatriv}
\nabla^L=d+\frac{1}{2}R^L(w,.)+O(|w|^2)\,.
\end{equation}
Using this formula together with the fact that $\varphi^L$ preserves
$\nabla^L$, we get in our coordinates a smooth bounded function 
$\lambda_{x_0}$ of $u\in\,]-\epsilon,\epsilon[$ and
$w\in B^{N}_{x_0}(0,\epsilon)$ such that
\begin{equation}
\varphi_{t_j,p}^{-1}
=e^{-2\pi\i p\lambda_j}\exp(p|w|^3\lambda_{x_0}(u,w))
~~\text{over}~~
]-\epsilon,\epsilon[\,\times\,V_j(\epsilon)\,.
\end{equation}
%Recall that $\varphi:X\fl X$ has only isolated fixed points by assumption, so that as $X$ is compact, the fixed point set $X^{\varphi}\subset X$ of $\varphi$ is finite, and there is $m\in\N$ such that $X^\varphi=\{x_1,\dots,x_m\}$, where $x_j\in X$ for any $1\leq j\leq m$.
We can then apply \cref{BTasy} in a chart of the form \cref{chart}
around $(u,x_0)$ and containing $W_u\subset T_{x_0}\Sigma\subset
T_{(u,x_0)}X$
via the identifications \cref{coordTSig} and \cref{idgal}, to get from
\cref{Gfla1} as $p\fl+\infty$,
%to get for all
%$y\in Y_{t_j}$, $w\in B^N_y(0,\epsilon)$, $t\in ]-\epsilon,\epsilon[$
%and $u\in\,]-\epsilon,\epsilon[$.
\begin{multline}\label{Gfla1bis}
I_{j,p}(t,u,w)=p^n
e^{-2\pi\i p\lambda_{j}}g(t_j+t)
\exp(p|w|^3\lambda_{x_0}(u,w))
e^{2\i\pi p tu}\rho(u,w)\\
\varphi^{E,-1}_{t_j+t}\sum_{r=0}^{k-1} p^{-\frac{r}{2}}
G_r\cT_{t_j+t,(u,x_0)}(\sqrt{p}\varphi_{u,t_j+t}(w),
\sqrt{p}w)+ p^{n} O(p^{-\frac{k}{2}+\delta})\,,
\end{multline}
for all $t,\,u\in\R$
with $|t-t_j|,\,|u|<\epsilon p^{-\theta/2}$ and
$w\in B^N_{x_0}(\epsilon p^{-\theta/2})$.
Following \cref{Taylorexp}, we
apply the Taylor expansion described in \cref{Taylordef}
to the right hand side of
\cref{Gfla1bis} in $t,\,u$ and $w$,
inside and outside the exponential of
the local model \cref{locmodT} for $\cT_{t_j+t,(u,x_0)}$.
We then get a family
$\{F_{r,x_0}(t,u,w)\in E_{t_j,x_0}\otimes E_{0,x_0}^*\}_{r\in\N}$
of polynomials in
$w\in N_{x_0}$ and $t,\,u\in\R$, of the
same parity as $r$, depending smoothly in
$x_0\in Y_{j}$ and with
$F_{0,x_0}(t,u,w)=\bar{\mu}^{-1}_{t_j}(x_0)\tau_{t_j,x_0}^E$
for all $t,\,u$ and $w$,
such that for any $k\in\N$, there is $M_k\in\N^*$
such that as $p\fl+\infty$,
\begin{multline}\label{Gfla2}
I_{j,p}(t,u,w)
=p^ne^{-2\pi\i p\lambda_{j}}g(t_j)
|\xi_{f,x_0}|^{-1}_{g^{TX}_0}
\varphi^{E,-1}_{t_j,x_0}\sum_{r=0}^{k-1} p^{-\frac{r}{2}}
F_{r,x_0}(\sqrt{p}t,\sqrt{p}u,\sqrt{p}w)\\
e^{2\i\pi p tu}\cT_{t_j,x_0}(\sqrt{p}t\xi_{f,x_0}+
\sqrt{p}d\varphi_{t_j}.w,\sqrt{p}w)
+ p^{n-\frac{k}{2}+\delta} O(|\sqrt{p}t|^{M_k}+1)\,.
\end{multline}
%\exp\left[-\pi\<\Pi_0^{t_j}(t\xi+d\varphi.w-w),t\xi+d\varphi.w-w\>
%-\sqrt{-1}\pi\Om(t\xi+d\varphi.w,w)\right]
Using the non-degeneracy assumption of
\cref{nondegdef} together with the explicit formula \cref{locmodT}
for the local model, we see that \cref{Gfla2}
decreases exponentially in $t\in\R$ and $w\in N_{x_0}$, but not
in $u\in\R$. Now to get an exponential decrease in
$u\in\R$, we will make the change of
variables $t\mapsto t/\sqrt{p}$ and $u\mapsto u/\sqrt{p}$,
integrate first with respect to $t\in\R$ and then with respect to
$u\in\R$. With this in mind, from the local model \cref{locmodT},
we get for any $t,\,u\in\R$ and $w\in N_{x_0}$,
\begin{multline}\label{Gfla3}
e^{2\i\pi tu}\cT_{t_j,x_0}(t\xi_{f,x_0}+d\varphi_{t_j}.w,w)=
\cT_{t_j,x_0}(d\varphi_{t_j}.w,w)\\
\exp\left[-\pi t^2\<\Pi_0^{t_j}\xi_{f,x_0},\xi_{f,x_0}\>-\pi t
\<(\Pi_0^{t_j}+(\Pi_0^{t_j})^*)\xi_{f,x_0},d\varphi_{t_j}.w-w\>
+2\i\pi tu\right]\,.
\end{multline}
Using the classical formula for the Fourier transform
of Gaussian integrals, we compute
\begin{equation}\label{Gfla4}
\begin{split}
\int_\R\Big(\int_\R \exp\Big[-\pi t^2
\<\Pi_0^{t_j}&\xi_{f,x_0},\xi_{f,x_0}\>\\
&-\pi t\<(\Pi_0^{t_j}+(\Pi_0^{t_j})^*)\xi,d\varphi_{t_j}.w-w\>
+2\i\pi tu\Big]dt\Big)du\\
=\<\Pi_0^{t_j}\xi_{f,x_0},\xi_{f,x_0}\>^{-\frac{1}{2}}&\\
\int_\R
\exp&\left[-\frac{\pi}{4}
\frac{\left(2u+\sqrt{-1}\<(\Pi_0^{t_j}+(\Pi_0^{t_j})^*)\xi_{f,x_0}
,d\varphi_{t_j}.w-w\>
\right)^2}{\<\Pi_0^{t_j}\xi_{f,x_0},\xi_{f,x_0}\>}\right]du\\
=\<\Pi_0^{t_j}\xi_{f,x_0},\xi_{f,x_0}\>^{-\frac{1}{2}}&\int_\R
\exp\left[-\pi
\frac{u^2}{\<\Pi_0^{t_j}\xi_{f,x_0},\xi_{f,x_0}\>}\right]du
=1\,,
\end{split}
\end{equation}
for a suitable choice of square root in the middle terms.
%for some polynomials $H_{r,x_0}(u,w)$ of the same parity
%as $r\in\N$, with $H_{0,x_0}\equiv\bar{\mu}^{-1}_{t_j}(x_0)
%\tau_{t_j}^E$.
%From \cref{Gfla4}, we see that after integration with respect
%to $t$, we get an exponential decrease in $u$ in \cref{Gfla3}.
%We can then integrate in $u\in\R$ and
%compute the Gaussian type integral after
In particular, this computation shows that we
get an exponential decrease in $u\in\R$ after integration with respect
to $t\in\R$. Using successive integration by parts, this computation
readily generalizes to the case when the exponential
is multiplied by a polynomial $F_{r,x_0}(t,u,w)$
of the same parity as $r\in\N$, to get as a result a polynomial
$H_{r,x_0}(w)\in E_{t_j,x_0}\otimes E_{0,x_0}^*$ in $w\in N_{x_0}$
of the same parity as $r\in\N$, with
$H_{0,x_0}(w)=\bar{\mu}^{-1}_{t_j}(x_0)
\tau_{t_j,x_0}^E$ for all $w\in N_{x_0}$.
Thus from \cref{Gfla2}, for all $k\in\N$ we get $\delta_k\in\,]0,1[$
as in \cref{delta'} such that for all
$w\in B^N_{x_0}(\epsilon p^{-\theta/2})$ as $p\fl+\infty$,
\begin{multline}\label{Gpfla5}
\int_{-\epsilon p^{-\frac{\theta}{2}}}
^{\epsilon p^{-\frac{\theta}{2}}}
\int_{-\epsilon p^{-\frac{\theta}{2}}}
^{\epsilon p^{-\frac{\theta}{2}}}
\Tr\left[I_{j,p}(t,u,w)\right]dt\,du\\
=p^{-1}\int_{-\epsilon p^{\frac{1-\theta}{2}}}
^{\epsilon p^{\frac{1-\theta}{2}}}
\int_{-\epsilon p^{\frac{1-\theta}{2}}}
^{\epsilon p^{\frac{1-\theta}{2}}}
\Tr\left[I_{j,p}(p^{-1/2}t,p^{-1/2}u,w)\right]dt\,du\\
=p^{-1}\int_\R\left(\int_\R
\Tr\left[I_{j,p}(p^{-1/2}t,p^{-1/2}u,w)\right]\,dt\right)du
+O(p^{-\infty})\\
=p^{n-1}e^{-2\pi\i p\lambda_{j}}g(t_j)\sum_{r=0}^{k-1}
p^{-\frac{r}{2}}\Tr\left[\varphi^{E,-1}_{t_j,x_0}
H_{r,x_0}(\sqrt{p}w)\right]\\
\cT_{t_j,x_0}
(\sqrt{p}d\varphi_{t_j}.w,\sqrt{p}w)
+p^{n-1} O(p^{-\frac{k}{2}+\delta_k})\,.
\end{multline}
Working locally, we can suppose that $Y_{j}$ is orientable.
Let $dv_{Y_{j}}$ be the Riemannian volume form on
$Y_{j}$ induced by $g^{T\Sigma}$, let
$dw$ be the Euclidean volume form on the fibres of $(N,g^N)$
and let $\rho\in\cinf(B^N(\epsilon),\R)$ be
such that via the identification \cref{idBN} of $V_j(\epsilon)$ with
the ball bundle $B^N(\epsilon)$ over $Y_j$, we have
\begin{equation}
dv_{\Sigma}=\rho\,dw\,dv_{Y_{j}}\,~~\text{and}~~\rho(0,x)=1~~\text{for all}~~x\in Y_j\,.
\end{equation}
Through the change of variable
$w\mapsto w/\sqrt{p}$, using the exponential decrease
of \cref{Gpfla5} in $w\in N_{x_0}$ and
taking the Taylor expansion in $w\in N_{x_0}$ of $\rho(w,x_0)$
as described in \cref{Taylordef} for all $x_0\in Y_j$,
we then get
polynomials $K_{r,x_0}\in\C[w]$ of the same parity as $r\in\N$,
with $K_{0,x_0}(w)=\bar{\mu}^{-1}_{t_j}(x_0)
\Tr\left[\varphi^{E,-1}_{t_j,x_0}\tau_{t_j,x_0}^E \right]$,
such that using \cref{Gpfla5}, we can rewrite
\cref{Gpfla0} as
\begin{multline}\label{Gpfla6}
\Tr\big[\hat{g}(pQ_p(f))\big]=p^{(\dim Y_{j}-1)/2}\Big[
\int_{x\in Y_{j}}\Big(\int_{B^N_{x}(\epsilon
p^{(1-\theta)/2})}
e^{-2\pi\i p\lambda_{j}}g(t_j)\\
\rho(p^{-1/2}w,x)\sum_{r=0}^{k-1}
p^{-\frac{r}{2}}K_{r,x}(w)\cT_{t_j,x}(d\varphi_{t_j}.w,w)dw\Big)\,dv_{Y_{j}}(x)+O(p^{-\frac{k}{2}+\delta_k'})\,\Big]\\
=p^{(\dim Y_{j}-1)/2}\Big[
e^{-2\pi\i p\lambda_j}g(t_j)\sum_{q=0}^{\lfloor \frac{k-1}{2}\rfloor}
p^{-q}\\
\int_{x\in Y_j}\Big(\int_{N_x}
K_{2q,x}(w)\cT_{t_j,x}(d\varphi.w,w)dw\Big)\,
\frac{dv_{Y_j}}{|\xi_{f}|_{g^{TX}_0}}(x)
+O(p^{-\frac{k}{2}+\delta_k'})\Big]\,,
\end{multline}
for some $\delta_k'\in\,]0,1[$ satisfying $\delta_k'\fl\delta$ as
$\theta\fl 1$,
where all the terms with $r$ odd vanished
in the same way as in \cref{odd} by the odd parity of $K_{2q+1,x}$
for all $q\in\N$.
This shows the expansion \cref{gutzexp}, and formula \cref{b0gutz}
follows from \cref{Gpfla6} and the second equality of
\cref{nu0}. Finally, formula \cref{Weylterm}
follows from the fact that function
$\mu_t\in\cinf(X,\C)$ of \cref{g0t} is constant equal to $1$
for $t=0$ and the fact that the vector field $v\in\cinf(f^{-1}(c),TX)$
defining the volume form \cref{liouvilleham}
over $f^{-1}(c)$ can be chosen to
be $v=J_0\xi_f/|\xi_f|_{g^{TX}_0}^2$.
\end{proof}
%
%The first order term \cref{Weylterm} appears as soon as
%$0\in\Supp g$, and is commonly called the \emph{Weyl term}
%of the Gutzwiller trace formula. However, the relevance of this formula
%for quantum chaos lies in the terms associated with
%\emph{isolated periodic orbits}.

%Let us first introduce the following
%definition.
%
%\begin{defi}\label{adaptJdef}
%An almost complex structure $J\in\End(TX)$ is said to be
%\emph{adapted} to $f\in\cinf(X,\R)$ at $I\subset\R$ if
%\begin{equation}\label{adaptJfla}
%d\varphi_t.\,J\xi_f=J\xi_f~~\text{over}~~f^{-1}(I)
%~~\text{for all}~~t\in\R\,,
%\end{equation}
%where $\xi_f\in\cinf(X,TX)$ is the Hamiltonian vector field
%of $f$ as in \cref{hamiltonian}.
%\end{defi}
%
%For any regular value $c\in I$ of $f$ and any interval $I\subset\R$
%containing $c$ small enough,
%it is easy to construct an almost complex structure $J\in\End(TX)$
%compatible with $\om$ adapted to $f$ at $I$,
%using partitions of unity and
%the fact that $d\varphi_t.\,\xi_f=\xi_f$ for all
%$t\in\R$.

Under the assumptions of \cref{gutzwillerperorbintro}, consider
now the case
when $Y_j\subset\Sigma$ satisfy $\dim Y_j=1$ for some
$1\leq j\leq m$, so that
$Y_j=\{\varphi_s(x)\}_{0\leq s< t(Y_j)}$,
for some $x\in f^{-1}(c)$ satisfying
$\varphi_t(x)=x$. Here $t(Y_j)>0$ is the smallest time
$t>0$ for which $\varphi_t(x)=x$, called
the \emph{primitive period} of $Y_j$.
We then get the following special case of \cref{gutzwiller},
recovering the explicit geometric term
associated with isolated periodic orbits in the Gutzwiller
trace formula of \cref{b0gutzperorbintroth}.

\begin{theorem}\label{gutzwillerperorb}
Consider the hypotheses of \cref{gutzwillerperorbintro},
and let $1\leq j\leq m$ be such that $\dim Y_j=1$
and such that $[J_0\xi_f,\xi_f]=0$ over $Y_j$.
Then
the term $b_{j,0}\in\C$ of \cref{b0gutz} is given by
%, depending only on $g$ around $0$ and on all the local data around $Y$, 
%such that for any $k\in\N$ and all $p\in\N$, we have
%\begin{equation}\label{gutzexp}
%\Tr\left[\hat{g}(pQ_p(f))\right]=\sum_{r=0}^{k-1}
%p^{-r} b_r+O(p^{-k})\,.
%\end{equation}
%Write $\{\gamma_j\}_{j=1}^m$
%for the set of periodic orbits in $f^{-1}(c)$ at time
%$\{t_j\in\Supp(g)\}_{j=1}^m$, the following formula
%for the first coefficient holds,
\begin{equation}\label{b0gutzper}
b_{j,0}=(-1)^{\frac{n-1}{2}}\int_{Y_j}
\frac{(\varphi^{K_X,-1}_{t_j}\tau^{K_X}_{t_j}
)^{-\frac{1}{2}}\Tr[\varphi^{E,-1}_{t_j}\tau^{E}_{t_j}]}
{|\det{}_{N}(\Id_{N}-d\varphi_{t_j}|_N)|^{1/2}}\,
\frac{|dv|_{Y_j}}{|\xi_{f}|_{g^{TX}_0}}\,,
\end{equation}
for some natural choices of square roots.
If $X$ admits a metaplectic structure \cref{metafla}
and taking $E=K_X^{1/2}$
to be the associated metaplectic correction, then formula
\cref{b0gutzper} becomes
\begin{equation}\label{b0gutzpermeta}
b_{j,0}=(-1)^{\frac{n-1}{2}}
\frac{t(Y_j)}
{|\det_{N_x}(\Id_{N}-d\varphi_{t_j}|_N)|^{1/2}}
\,,
\end{equation}
not depending on $x\in Y_j$.
\end{theorem}
\begin{proof}
The assumption $[J_0\xi_f,\xi_f]=0$ over $Y_j$ means that
$d\varphi_t.\,J\xi_f=J\xi_f$ over $Y_j$
for all $t\in\R$. As $d\varphi_{t}.\xi_f=\xi_f$
by definition and as $d\varphi_t$ preserves
the symplectic form $\om$, by definition \cref{gTXintro}
of $g^{TX}_0$, this implies that $d\varphi_t$
preserves the normal bundle $N\subset T\Sigma$ of $Y_j$
inside $\Sigma$, for all $t\in\R$.
We are then under the assumptions of \cref{loctrcT}, and
formula \cref{b0gutzper} is a consequence of \cref{gutzwiller}.
%
%In particular, we get the following formula
%for $a_{0,j}\in\cinf(Y_j,\C)$
%in formula \cref{Gpfla6} for all $x\in Y_j$,
%\begin{equation}\label{a0y}
%a_{j,0}(x)=(-1)^{n/2}
%\frac{(\varphi_{t_j,x}^{K_X,-1}\tau_x^{K_X}
%)^{-1/2}\Tr_{E_x}[\varphi^{E,-1}_{t_j}\tau^{E}_{t_j}]}
%{|\det_{N_x}(d\varphi|_N-\Id_N)|^{\frac{1}{2}}}
%\end{equation}
%for some natural choice of square roots.

Note now that from
the $1$-parameter group property of $\varphi_t$,
for any $t\in\R$ and $x\in Y_j$, we have
\begin{equation}
d\varphi_{t_j,\varphi_t(x)}=d\varphi_{t,x}d\varphi_{t_j,x}
d\varphi_{t,x}^{-1}\,.
\end{equation}
This shows that the quantity
$\det{}_{N_x}(\Id_{N}-d\varphi_{t_j}|_N)$ is actually independent
of $x\in Y_j$. Considering the case $E=K_X^{1/2}$ and
recalling that $\xi_f$ is the tangent vector field of the curve
$Y_j$, we get
\begin{equation}
\int_{Y_j}(\varphi_{t_j}^{K_X,-1}\tau_{t_j}^{K_X}
)^{-\frac{1}{2}}\Tr_{E}[\varphi^{E,-1}_{t_j}\tau^{E}_{t_j}]\frac{|dv|_{Y_j}}
{|\xi_f|_{g^{TX}_0}}=\int_{Y_j}\frac{|dv|_{Y_j}}{|\xi_f|_{g^{TX}_0}}
=t(Y_j)
\,.
\end{equation}
This gives formula \cref{b0gutzpermeta}.
\end{proof}
%
%Note that the hypotheses of \cref{gutzwillerperorbintro}
%described in the introduction actually imply the hypotheses of
%\cref{gutzwillerperorb}, since by an elementary transversality
%argument, an isolated periodic orbit on $f^{-1}(c)$
%of a non-degenerate flow along $f^{-1}(c)$ always comes in a
%smooth family of isolated periodic orbits on $f^{-1}(u)$ parametrized by
%$u\in\,]-\epsilon,\epsilon[$ for some $\epsilon>0$.

Let us now show how these results can be applied to
contact topology. To that end, consider $f\in\cinf(X,\R)$
and an almost complex structure
$J\in\End(TX)$
%adapted to $f\in\cinf(X,\R)$
%at an open interval $I\subset\R$
%of regular values of $f$, and
satisfying
%such that
%
%
%let us introduce the following
%definition.
%
%\begin{defi}\label{contacttypedef}
%A time-independent Hamiltonian $f\in\cinf(X,\R)$ is said to
%be of \emph{contact type} at a regular value $c\in\R$ if there
%is a neighborhood $U\subset X$ of $f^{-1}(c)$ such that
\begin{equation}\label{contactfla}
d\iota_{J\xi_f}\om=\om~~\text{over}~~f^{-1}(I)\,,
\end{equation}
over some interval $I\subset\R$ of regular values of $f$.
%\end{defi}
%Then we have an identification
%\begin{equation}\label{idgal}
%\Psi:
%f^{-1}(]-\epsilon,\epsilon[)\xrightarrow{~~\sim~~}\,
%]-\epsilon,\epsilon[\,\times\,\Sigma\,,
%\end{equation}
%sending $f^{-1}(u)$ on $\{u\}\times\Sigma$ for all
%$u\in\,]c-\epsilon,c+\epsilon[$ as in \cref{idgal}, such that
%\begin{equation}
%d\Psi.J\xi_f=\D{u}\,.
%\end{equation}
Then considering the Riemannian metric $g^{TX}=\om(\cdot,J\cdot)$
as in \cref{gTXintro}, the form
\begin{equation}\label{contactform}
-\frac{\iota_{J\xi_f}\om}{|\xi_f|_{g^{TX}}^2}\in\Om^1(X,\R)
\end{equation}
restricts to
a \emph{contact form} $\alpha\in\Om^1(\Sigma,\R)$
over $\Sigma:=f^{-1}(c)$ for any $c\in I$, meaning that
$\alpha\wedge d\alpha^{n-1}$ is a volume form over $\Sigma$.
%so that it realizes the identification $\Psi_t$ of \cref{idgal},
%for which $f(u,x)=u$ for all $(u,x)\in\R\times\Sigma=S\Sigma$
%with $|u|<\epsilon$ for some $\epsilon>0$.
%Then $(f^{-1}(I),\om)$ can be symplectically identified
%as a neighbourhood of $\{0\}\times\Sigma$
The restriction of the Hamiltonian
vector field $\xi_f\in\cinf(X,TX)$ over $f^{-1}(c)$
induces the \emph{Reeb vector field} of $(\Sigma,\alpha)$,
which is the unique vector field $\xi\in\cinf(\Sigma,T\Sigma)$
satisfying
\begin{equation}\label{Reebdef}
\left\{
\begin{array}{l}
  \iota_\xi \alpha=1\,, \\
  \\
  \iota_\xi d\alpha = 0\,.
\end{array}
\right.
\end{equation}
The corresponding \emph{Reeb flow} is the flow of diffeomorphisms
$\varphi_t:\Sigma\fl\Sigma$ generated by $\xi$, and its periodic
orbits are called the \emph{Reeb orbits} of $(\Sigma,\alpha)$.
An isolated Reeb orbit of period $t_0\in\R$
is said to be \emph{non-degenerate} if it satisfies \cref{nondegdef}
as a fixed point set of $\varphi_{t_0}$ inside $\Sigma$.

Conversely, given a compact manifold $\Sigma$ endowed with
a contact form $\alpha\in\Om^1(\Sigma,\R)$, we define
a symplectic form $\om^{S\Sigma}\in\Om^2(S\Sigma,\R)$
over $S\Sigma:=\R\times\Sigma$ by the formula
\begin{equation}\label{symplectization}
\begin{split}
\om^{S\Sigma}:=-d(e^{u}\pi^*\alpha)\,,
~~\quad\text{with}~~\quad
\pi:S\Sigma:=\R\times\Sigma&\longrightarrow\Sigma
\,.\\
(u,x)&\longmapsto x
\end{split}
\end{equation}
The symplectic manifold $(S\Sigma,\om^{S\Sigma})$ is called
the \emph{symplectization} of $(\Sigma,\alpha)$. Consider the function
$f\in\cinf(S\Sigma,\R)$ defined by
\begin{equation}\label{fu=u}
f(u,x):=e^u~~\text{for all}~~(u,x)\in\R\times\Sigma\,.
\end{equation}
Then its Hamiltonian
vector field $\xi_f\in\cinf(S\Sigma,TS\Sigma)$
restricts over any level set of $f$ to the
Reeb vector field $\xi\in\cinf(\Sigma,T\Sigma)$ of $(\Sigma,\alpha)$.
On the other hand, consider
an almost complex structure $J\in\End(TS\Sigma)$
compatible with $\om^{S\Sigma}$ satisfying
\begin{equation}\label{Jxi=du}
J\xi_f=\D{u}\,,
\end{equation}
in the coordinates $(u,x)\in\R\times\Sigma=S\Sigma$.
Such an almost complex structure always exists.
Finally, consider the situation when an open set of the form
$I\times\Sigma\subset S\Sigma$
can be symplectically embedded
into a compact prequantized symplectic manifold $(X,\om)$ without
boundary.
We can then extend the Hamiltonian and the almost complex
structure defined by \cref{fu=u} and \cref{Jxi=du}
to the whole of $(X,\om)$ using partitions of unity.
The typical case when such an embedding
exists is when $\Sigma$ is the boundary of a star-shaped domain
in $\R^{2n}$, with contact form $\alpha\in\Om^1(\Sigma,\R)$
given by the restriction of the standard \emph{Liouville form}
$\lambda\in\Om^1(\R^{2n},\R)$. Then using a Darboux chart,
an open set of the form
$I\times\Sigma\subset S\Sigma$ can always
be symplectically embedded in $(X,\om)$.
More generally, by the results of
\cite[Th.\,1.3]{EH02} and
\cite[Cor.\,1.11]{Laz18}, any \emph{fillable} contact manifold
satisfies this property.
The following result then shows how to use the above picture to
extract informations on isolated non-degenerate Reeb orbits
of $(\Sigma,\alpha)$ from the geometric quantization of $(X,\om)$.

\begin{prop}\label{contactcor}
Let $\Sigma$ be a compact manifold without boundary
endowed with a contact form $\alpha\in\Om^1(\Sigma,\R)$, and assume
that an open set of the form $I\times\Sigma$ in its symplectization
$(S\Sigma,\om^{S\Sigma})$
can be symplectically embedded in a compact prequantized
symplectic manifold $(X,\om)$ without boundary.
Consider the Hamiltonian $f\in\cinf(X,\R)$ and
the almost complex structure $J\in\End(TX)$ defined via \cref{fu=u}
and \cref{Jxi=du}.

Then for any $c\in I$, \cref{gutzwillerperorbintro}
holds as soon as the fixed point set of the Reeb flow
$\varphi_t:\Sigma\fl\Sigma$ is non-degenerate for all $t\in\Supp g$,
and \cref{b0gutzperorbintroth} holds for the
non-degenerate isolated Reeb orbits of $(\Sigma,\alpha)$
in case $(X,\om)$ is endowed with a metaplectic structure.
\end{prop}
\begin{proof}
%Note that up to multiplication by
%$\lambda$ on the first component of
%the symplectization $S\Sigma=\R\times\Sigma$ of
%\cref{symplectization},
%the identification of $f^{-1}(I)\subset X$
%with an open set of the symplectization $S\Sigma$ of
% coincides with the
%identification \cref{idgal}. By the second formula of \cref{fu=u},

Recall that the Hamiltonian vector field
$\xi_f\in\cinf(S\Sigma,TS\Sigma)$ of
$f\in\cinf(S\Sigma,\R)$ defined by \cref{fu=u}
restricts to the Reeb vector field $\xi\in\cinf(\Sigma,T\Sigma)$
of $(\Sigma,\alpha)$ over $\{u\}\times\Sigma$, for all $u\in\R$.
Then the Hamiltonian
flow $\varphi_t:X\fl X$ of $f$ over $I\times\Sigma$
is of the form $\varphi_t(u,x)=(u,\varphi_t(x))$ for all
$(u,x)\in I\,\times\,\Sigma$. It thus satisfies the hypotheses of
\cref{gutzwillerperorbintro} as long as the fixed point set of
the Reeb flow $\varphi_t:\Sigma\fl\Sigma$ is non-degenerate
for all $t\in\Supp g$, the discreteness of $T\subset\Supp g$
following from the fact that $d\varphi_t$ preserves
$\Ker\alpha\subset T\Sigma$ for all $t\in\R$,
so that by \cref{Reebdef} there is no $v\in\Ker\alpha_x$
such that $\xi_x+d\varphi_{t}.v=v$ for any $x\in\Sigma$
satisfying $\varphi_t(x)=x$.
On the other hand, the formula \cref{Jxi=du} shows
\begin{equation}
[J\xi_f,\xi_f]=\Big[\D{u}\,,\,\xi\Big]=0~~\text{over}
~~\R\times\Sigma\,,
\end{equation}
so that the hypotheses of \cref{b0gutzperorbintroth} are
satisfied as well. This shows the result.
\end{proof}

%\section{fantome}
\bibliographystyle{amsplain}
%\nocite{*}
\providecommand{\bysame}{\leavevmode\hbox to3em{\hrulefill}\thinspace}
\providecommand{\MR}{\relax\ifhmode\unskip\space\fi MR }
\providecommand{\MRhref}[2]{%
  \href{http://www.ams.org/mathscinet-getitem?mr=#1}{#2}
}
\providecommand{\href}[2]{#2}

\ \\
Tel Aviv University - School of Mathematical Sciences,\\
Ramat Aviv, Tel Aviv 69978, Israël\\
\\
\emph{E-mail adress}: louisioos@mail.tau.ac.il
\end{document}